\documentclass{amsart}

\usepackage[top=3.4cm, bottom=3.4cm, margin=1.3in]{geometry}

\usepackage[title]{appendix}
\usepackage{amssymb}
\usepackage{amsmath}
\usepackage{amsthm}
\usepackage{tikz-cd}
\usepackage{mathtools}
\usepackage{comment}

\usepackage{color}
\usepackage{enumerate}
\usepackage{hyperref}
\usepackage[capitalize]{cleveref}

\newcommand{\wh}{\widehat}
\newcommand{\wt}{\widetilde}

\newcommand{\Q}{\mathbb{Q}}
\newcommand{\R}{\mathbb{R}}

\newcommand{\Z}{\mathbb{Z}}

\DeclareMathOperator{\Hom}{Hom}
\DeclareMathOperator{\id}{id}

\DeclareMathOperator{\rank}{rank}

\DeclareMathOperator{\Res}{Res}

\DeclareMathOperator{\IM}{im}
\DeclareMathOperator{\Tor}{Tor}
\DeclareMathOperator{\Ker}{ker}
\DeclareMathOperator{\coker}{coker}

\DeclareMathOperator{\SO}{SO}
\DeclareMathOperator{\Tors}{Tors}

\newtheorem{thm}{Theorem}[section]

\newtheorem*{thm*}{Theorem}
\newtheorem{prop}[thm]{Proposition}
\newtheorem*{prop*}{Proposition}
\newtheorem{lemma}[thm]{Lemma}
\newtheorem{corollary}[thm]{Corollary}
\newtheorem*{corollary*}{Corollary}

\newtheorem*{question*}{Question}

\theoremstyle{definition}
\newtheorem{definition}[thm]{Definition}

\newtheorem{convention}[thm]{Convention}

\theoremstyle{remark}
\newtheorem{remark}[thm]{Remark}
\newtheorem*{remark*}{Remark}

\numberwithin{equation}{section}
\newcounter{commentcounter}

\renewcommand{\odot}{\,\square\,}

\title{Homotopy classification of 4-manifolds whose fundamental group is dihedral}

\author{Daniel Kasprowski}
\address{Rheinische Friedrich-Wilhelms-Universit\"{a}t Bonn, Mathematisches Institut, \newline \indent
Endenicher Allee 60, 53115 Bonn, Germany}
\email{\href{mailto:kasprowski@uni-bonn.de}{kasprowski@uni-bonn.de}}

\author{John Nicholson}
\address{Department of Mathematics, UCL, Gower Street, London, WC1E 6BT, United Kingdom}
\email{\href{mailto:j.k.nicholson@ucl.ac.uk}{j.k.nicholson@ucl.ac.uk}}

\author{Benjamin Ruppik}
\address{Max-Planck-Institut f\"{u}r Mathematik, Vivatsgasse 7, 53111 Bonn, Germany}
\email{\href{mailto:bruppik@mpim-bonn.mpg.de}{bruppik@mpim-bonn.mpg.de}}

\makeatletter
\@namedef{subjclassname@2020}{%
  \textup{2020} Math. Subj. Class}
\makeatother

\subjclass[2020]{
	Primary 57K40; 
	Secondary 
	16E05, 
	57N65, 
	57P10. 
	\hfill \today
	}

\begin{document}

\begin{abstract}
We show that the homotopy type of a finite oriented Poincar\'{e} 4-complex is determined by its quadratic 2-type provided its fundamental group is finite and has a dihedral Sylow 2-subgroup.
By combining with results of Hambleton--Kreck and Bauer, this applies in the case of smooth oriented 4-manifolds whose fundamental group is a finite subgroup of $\SO(3)$. An important class of examples are elliptic surfaces with finite fundamental group.
	\vspace{-7mm}
\end{abstract}

\maketitle

\section*{Introduction}

Recall that a finite oriented Poincar\'{e} 4-complex is a finite
CW-complex with a fundamental class $[X] \in H_4(X;\Z)$ such that
\[- \cap [X] \colon C^{4-*}(X;\Z[\pi_1(X)]) \to C_*(X;\Z[\pi_1(X)])\]
is a chain homotopy equivalence \cite{Wall67}.
Every closed topological 4-manifold has the structure of a finite Poincar\'{e} 4-complex,
but there are finite Poincar\'{e} 4-complexes which are not homotopy equivalent
to any closed topological 4-manifold \cite{HM78}.

In 1988, Hambleton and Kreck \cite[Theorem 1.1]{HK} proved that an oriented 
Poincar\'e 4-complex $X$ with finite fundamental group $\pi_1(X)$ is determined up to homotopy equivalence by three invariants, including
the isometry class of its quadratic 2-type, i.e.\ the quadruple
\[
[\pi_1(X),\pi_2(X),k_X,\lambda_X]
\] 
where $\pi_2(X)$ is considered as a $\Z[\pi_1(X)]$-module,
$k_X\in H^3(\pi_1(X);\pi_2(X))$ is the $k$-invariant
determining the Postnikov $2$-type of $X$
and $\lambda_X$ is the equivariant intersection form on $\pi_2(X)$.

Moreover, for oriented Poincar\'e 4-complexes whose fundamental group has $4$-periodic cohomology, the quadratic $2$-type is actually a complete homotopy type invariant (see \cite[Theorem A]{HK}). 
%
This was improved upon by Bauer \cite{Ba88} who showed this was true
under the weaker assumption that $\pi_1(X)$ is a finite group whose Sylow 2-subgroup
has 4-periodic cohomology,
i.e.\ is isomorphic to a cyclic group $\Z/2^n$
or a generalised quaternion group $Q_{2^n}$.

Recently, it was shown in \cite{KPR20} that this is also true
when the Sylow 2-subgroup of $\pi_1(X)$ is abelian with two generators,
i.e.\ of the form $\Z/2^n \times \Z/2^m$.
The aim of this article will be to extend this to the case
where the Sylow 2-subgroup of $\pi_1(X)$ is dihedral, i.e. is isomorphic to the dihedral group $D_{2^n}$ of order $2^n$ for some $n \ge 2$.

\begingroup
\renewcommand\thethm{\Alph{thm}}
\begin{thm}
	\label{thm:main}
	Let $\pi$ be a finite group whose Sylow $2$-subgroup is dihedral. Then the homotopy type of a finite oriented Poincar\'{e} $4$-complex with fundamental group $\pi$ is determined by the isometry class of its quadratic $2$-type. That is, every isometry of the quadratic 2-types of $M$ and $N$ is realized by a homotopy equivalence $M\to N$.
\end{thm}
\endgroup
By \cite[Theorem 1.1~and Remark~1.2]{HK} and Teichner~\cite{teichnerthesis} (see \cite[Corollary~1.6]{KT}),  
in order to prove that the homotopy type of a finite oriented Poincar\'{e} $4$-complex $X$ is determined by its quadratic 2-type, it suffices to show that $\Z\otimes_{\Z[\pi_1(X)]}\Gamma(\pi_2(X))$ is torsion free as an abelian group where $\Gamma$ denotes Whitehead's quadratic functor (see \Cref{sec:prelim}).  

For a finitely presented group $\pi$  and $n \ge 1$, recall that the
\textit{$n$th stable syzygy} $\Omega_n(\Z)$,
which we also write as $\Omega_n^{\pi}(\Z)$, is the set of $\Z \pi$-modules $J$
for which there exists an exact sequence
\[
	0 \to J \to F_{n-1} \to \cdots \to F_0 \to \Z \to 0
\]
where the $F_i$ are finitely generated free $\Z \pi$-modules. 
It follows from \cite{Ba88} (see also \Cref{cor:torsionfree} and \Cref{lemma:sylow-2}) that $\Z\otimes_{\Z[\pi_1(X)]}\Gamma(\pi_2(X))$ is torsion free provided $\Z\otimes_{\Z\pi} \Gamma(J)$ and $\Z\otimes_{\Z\pi} \Gamma(J^*)$ are torsion free for some $J \in \Omega_3^{\pi}(\Z)$ where $\pi$ is the Sylow 2-subgroup of $\pi_1(X)$ and $J^*=\Hom_\Z(J,\Z)$ is the dual of $J$.

In order to prove that $\Z\otimes_{\Z[\pi_1(X)]}\Gamma(\pi_2(X))$ is torsion free, it therefore suffices to accomplish the following two tasks for the finite 2-group $\pi$ which arises as the Sylow 2-subgroup of $\pi_1(X)$:

\begin{enumerate}
	\item Find a explicit parametrisation for $\Omega_3^{\pi}(\Z)$,
	i.e.\ give an explicit description of a $\Z \pi$-module $J$
	such that $J \in \Omega_3^{\pi}(\Z)$
	\item Show that $\Tors(\Z\otimes_{\Z\pi} \Gamma(J))=0$
	and $\Tors(\Z\otimes_{\Z\pi} \Gamma(J^*))=0$.	
\end{enumerate}

Recall that, if $K$ is a finite 2-complex with fundamental group $\pi$, then $\pi_2(K) \in \Omega_3^{\pi}(\Z)$. It is still an open problem, though it is a consequence of an affirmative solution to Wall's D2 problem \cite{Jo03-book}, to determine whether or not every $J \in \Omega_3^{\pi}(\Z)$ arises as $\pi_2(K)$ for a finite 2-complex $K$ with fundamental group $\pi$.
It is therefore not surprising that the existing literature on Wall's D2 problem contains many computations of $\Omega_3^{\pi}(\Z)$ (see \cite{Jo12}). 

More specifically, the case of dihedral groups was explored by Mannan and O'Shea \cite{MO13}, and also independently by Hambleton \cite{Ha19} building upon earlier work with Kreck \cite{HK93}. Both sources contain suitable parametrisations for $\Omega_3^{\pi}(\Z)$ albeit of different forms.

After recalling basic facts about Whitehead's $\Gamma$ functor and Tate cohomology in \Cref{sec:prelim}, we will then give an overview of the theory of syzygies of finite groups in \Cref{sec:syz}. In \Cref{sec:ker1}, we will make use of the result of Hambleton--Kreck \cite{HK93} to obtain an explicit parametrisation for some $J \in \Omega_3^{\pi}(\Z)$ in the case where $\pi = D_{4n}$ is the dihedral group of order $4n$, and \Cref{sec:ker2} will then be dedicated to the proof that $\Tors(\Z\otimes_{\Z\pi} \Gamma(J))=0$. In \Cref{sec:coker1}, we will obtain an explicit parametrisation for $J^*$ and, finally, in \Cref{sec:coker2} we will prove also that $\Tors(\Z\otimes_{\Z\pi} \Gamma(J^*))=0$.

\vspace{4.22mm}


We conclude by noting that every finite subgroup of $\SO(3)$ has a cyclic or dihedral Sylow 2-subgroup. In particular, by combining our result with \cite{Ba88,HK}, we get that \cref{thm:main} also holds in the case where $\pi$ is a finite subgroup of $\SO(3)$. This makes possible a complete homotopy classification of 4-manifolds whose fundamental group is $\pi$. The study of these manifolds was one of the motivations for the original results of Hambleton--Kreck \cite{HK93} as they contain all elliptic surfaces with finite fundamental group (see, for example, \cite[p.~81]{HK93-III}). These were the subject of a subsequent paper  \cite{HK93-III} where they studied exotic smooth structures on elliptic surfaces.

Note also that, if $\pi$ is a fixed point free finite subgroup of $\SO(4)$, then $\pi$ has $4$-periodic cohomology and so the results of \cite{HK} imply that finite oriented Poincar\'{e} $4$-complexes with fundamental group $\pi$ are also determined by the isometry class of its quadratic $2$-type. 

However, it is not clear whether or not this holds for all finite subgroups of $\SO(4)$. 
For example, let $\pi = D_8 \times \Z/2$. Then $\pi$ is a finite subgroup of $\SO(4)$ since it is contained in the central product $Q_8 \circ Q_8$ as the image of $Q_8 \times Q_8$ under the double cover $S^3 \times S^3 \to \SO(4)$. On the other hand, if $J \in \Omega_3^{\pi}(\Z)$, then it follows from computations of the third author \cite{ruppik} and Hennes \cite{hennes} that $\Tors(\Z\otimes_{\Z\pi} \Gamma(J))=0$ and $\Tors(\Z\otimes_{\Z\pi} \Gamma(J^*))\ne 0$.
For a finite 2-complex $K$ with $\pi_1(K) \cong \pi$, let $X$ be the boundary of a smooth regular neighbourhood of an embedding of $K$ in $\R^5$.
Then $X$ is a 4-manifold with $\pi_1(X) \cong \pi$ and $\pi_2(X) \cong J_0 \oplus J_0^*$ where $J_0 = \pi_2(K) \in \Omega_3^{\pi}(\Z)$ \cite[p.~95]{HK}. It follows that $\Tors(\Z \otimes_{\Z \pi} \pi_2(X)) \ne 0$ and so the proof of \cref{thm:main} does not extend to this case. 

It is still not known whether or not the homotopy type of a finite oriented Poincar\'{e} 4-complex with arbitrary finite fundamental group $\pi$ is determined by the isometry class of its quadratic 2-type, though we do not expect this to be true when $\pi = D_8 \times \Z/2$ (as above) or $\pi = (\Z/2)^3$ (as discussed in \cite{KPR20}).
In the case where $X$ is non-orientable, this was shown by Kim, Kojima and Raymond \cite{KKR92} to be false even for smooth 4-manifolds in the case $\pi = \Z/2$.

\subsection*{Acknowledgement.} 
DK was funded by the Deutsche Forschungsgemeinschaft under Germany's Excellence Strategy -- GZ 2047/1, Projekt-ID 390685813.
JN was supported by the UK Engineering and Physical Sciences Research Council (EPSRC) grant EP/N509577/1.
BR was supported by the Max Planck Institute for Mathematics in Bonn.
We would like to thank Mark Powell for useful discussions, Ian Hambleton for helpful comments on the introduction and an anonymous referee for their careful reading.

\section{Preliminaries}
\label{sec:prelim}

The aim of this section will be to define Whitehead's $\Gamma$ functor and Tate homology, and recall a few of their basic properties which we will use in the rest of the article. From now on, all modules will be assumed to be finitely generated left modules.

The following was first defined by Whitehead in \cite{whitehead}.
\begin{definition}[$\Gamma$ groups]
	\label{def:gamma}
	Let $A$ be an abelian group. Then $\Gamma(A)$ is an abelian group with generators the elements of $A$.
	We write $a$ as $v(a)$ when we consider it as an element of $\Gamma(A)$. The group $\Gamma(A)$ has the following relations:
	\[
		\{v(-a)-v(a)\mid a\in A\}  \quad \text{ and }
	\]
	\[
		\{v(a+b+c)-v(b+c)-v(c+a)-v(a+b)+v(a)+v(b)+v(c)\mid a,b,c\in A\}.
	\]
	In particular, $v(0_A)=0_{\Gamma(A)}$.
\end{definition}

We will be interested in the case where $A$ is a free abelian group, in which case $\Gamma(A)$ has the following simple description.

\begin{lemma}[{\cite[page 62]{whitehead}}]
	\label{lem:gammafree}
	If $A$ is free abelian with basis $\mathfrak{B}$,
	then $\Gamma(A)$ is free abelian with basis
	\[
	\{ v(b), v(b+b')-v(b)-v(b') \mid b \neq b' \in \mathfrak{B} \}.
	\]
\end{lemma}

Recall that a \textit{$\Z \pi$-lattice} is a $\Z \pi$-module $A$
whose underlying abelian group is finitely generated torsion free,
and so is of the form $\Z^n$ for some $n \ge 0$.
For example, if $X$ is a finite oriented Poincar\'{e} $4$-complex
with finite fundamental group~$\pi$, then
\[
	\pi_2(X) \cong H_2(\wt{X};\Z) \cong H^2(\wt{X};\Z) \cong \Hom(H_2(\wt{X};\Z),\Z),
\]
is finitely-generated and torsion-free as an abelian group and so $\pi_2(X)$ is a $\Z \pi$-lattice.

If $A$ is a $\Z \pi$-lattice, then we can view $\Gamma(A)$ as a $\Z \pi$-module as follows. Firstly, by \Cref{lem:gammafree}, we can take $\Gamma(A)$ to be the subgroup of symmetric elements of $A\otimes A$ given by sending $v(a)$ to $a\otimes a$.
Observe that $v(b+b')-v(b)-v(b')$ corresponds to the symmetric tensor $b \otimes b' + b' \otimes b$. We can now let the group $\pi$ act on $\Gamma(A) \subseteq A \otimes A$ via
\[
	g\cdot \sum_i (a_i \otimes b_i) = \sum_i (g \cdot a_i) \otimes (g \cdot b_i).
\]
For $a, b \in A$, we will write
\[
	a \odot b 
	= a \otimes b + b \otimes a \in A \otimes A
\]
and we will also often write $a^{\otimes 2} = a \otimes a \in A \otimes A$
to shorten many expressions.
We will continue to use that $a \odot b = b \odot a$,
$a \odot a = 2 a \otimes a$ and
$a \odot b + c \odot b = (a + c) \odot b$ for $a, b, c \in A$.
For a map $f \colon A \rightarrow B$ of $\Z \pi$ modules we have induced
$f_* \colon \Gamma(A) \rightarrow \Gamma(B)$ with $f_*(a \otimes a) = f(a) \otimes f(a)$
and $f_*(a \odot b) = f(a) \odot f(b)$.

To compute $\Gamma$ groups we will make frequent use of the following lemma.
\begin{lemma}[{\cite[Lemma~4]{Ba88}}]
	\label{lem:bauer}
	Let $\pi$ be a group.
	If $0\to A\to B\to C\to 0$ is a short exact sequence of $\Z \pi$-lattices,
	then there exists a $\Z \pi$-lattice $D$ and short exact sequences of $\Z\pi$-modules
	\[
		0\to \Gamma(A)\to \Gamma(B)\to D\to 0
	\]
	and
	\vspace{-2mm}
	\[0\to A\otimes_\Z C\xrightarrow[]{f} D\to \Gamma(C)\to 0.\]
	If $\{a_i\},\{c_j\}$ and $\{a_i,\wt c_j\}$ are bases for $A$, $C$, and $B$ as free abelian groups respectively,
	where $\wt c_j$ is a lift of $c_j$, then the map $f$ is defined by
	\[
		f(a_i\otimes c_j) = [a_i\otimes \wt c_j +\wt c_j\otimes a_i] = [a_{i} \odot \wt c_j]
		\in D \cong \Gamma(B)/\Gamma(A).
	\]
\end{lemma}

\begin{remark}
	For the direct sum of $\Z \pi$-lattices $A$ and $B$,
	these short exact sequences split,
	and so $\Gamma(A \oplus B)
	\cong \Gamma(A) \oplus \Gamma(B) \oplus A \otimes_{\Z} B$.
\end{remark}

The second key definition we require is as follows. See \cite{brown} for a convenient reference.
\begin{definition}[Tate homology] \label{def:Tate}
	Given a finite group $\pi$ and a $\Z\pi$-module $A$,
	the \emph{Tate homology groups} $\widehat H_n(\pi;A)$
	are defined as follows.
	Let $N \colon A_\pi\to A^\pi$ denote multiplication with the norm element from the orbits $A_\pi:=\Z\otimes_{\Z\pi}A$ of $A$ to the $\pi$-fixed points of $A$, that is, $N(1 \otimes a) = \sum_{g \in \pi} g a$. This is well-defined since
	$N(1 \otimes ga) = N \cdot ga = N \cdot a = N(1 \otimes a).$
	Then
	\begin{align*}
	\widehat H_n(\pi;A)&:=H_n(\pi;A)\text{~for~}n\geq 1\\
	\wh H_0(\pi;A)&:=\ker(N)\\
	\wh H_{-1}(\pi;A)&:=\coker(N)\\
	\wh H_n(\pi;A)&:=H^{-n-1}(\pi;A)\text{~for~}n\leq -2
	\end{align*}
	We can similarly define Tate cohomology groups by, for example, letting $\wh H^{n}(\pi;A) = \wh H_{-n-1}(\pi;A)$.
\end{definition}

We will require the following properties of Tate homology, and we will use them throughout the article without further mention. 

\begin{lemma}[{\cite[VI~(5.1)]{brown}}]\label{lem:les-in-tate-homology}
	If $0\to A\to B\to C\to 0$ is a short exact sequence of $\Z\pi$-modules, then there is a long exact sequence of Tate homology groups
	\[\cdots\to \wh H_n(\pi;A)\to \wh H_n(\pi;B)\to \wh H_n(\pi;C)\to \wh H_{n-1}(\pi;A)\to\cdots\]
\end{lemma}

\begin{lemma}[{\cite[VI~(5.2)]{brown}}]
	\label{lem:free-implies-tate-zero}
	If $A$ is a free $\Z\pi$-module, then $\wh H_n(\pi;A)=0$ for all $n\in\Z$. 
\end{lemma}

For a $\Z \pi$-module $A$, let $\Tors(A)$ denote the torsion subgroup of $A$ as an abelian group. The following lemmas are elementary and we refer to \cite{KPR20} for proofs. 

\begin{lemma}[{\cite[Lemma~3.2]{KPR20}}]\label{lemma:torsion-gamma-equals-tate}
	If $\pi$ is a finite group and $A$ is a $\Z \pi$-lattice,
	then there is an isomorphism of abelian groups
	\[
		\wh H_0(\pi;A) \cong \Tors(\Z\otimes_{\Z\pi} A).
	\]
\end{lemma}

\begin{remark} 
	\label{rem:H_1}
	As an abelian group, we have $\Z \otimes_{\Z \pi} A \cong A/\pi$
	where $\pi$ acts on $A$ by left multiplication.
	We will therefore also often use $a \in A$ to refer to
	the element $1 \otimes a \in \Z \otimes_{\Z\pi} A$.
\end{remark}

While we defined $\wh H_n(\pi;A)$ as abelian groups in Definition \ref{def:Tate}, it will be useful to fix more explicit descriptions when $n = 0, \pm 1$. The following will be in place from now on.
\begin{convention}
	\label{convention}
	Throughout the rest of this article, $A$ will be a $\Z \pi$-lattice.
	Following \Cref{def:Tate} and \Cref{rem:H_1},
	we will use $a \in A$ to denote both elements of the homology groups:
	\begin{align*}
		\wh H_0(\pi;A) &= \Tors(\Z\otimes_{\Z\pi} A) = \Tors(A/\pi) \\
		\wh H_{-1}(\pi;A) &= \coker(N) = A^{\pi}/(N \cdot A_{\pi}).
	\end{align*}
	Furthermore, we will write:
	\[
		\wh H_1(\pi;A) = 
			\frac{\ker(d_1 \otimes \id_A \colon C_1 \otimes_{\Z \pi} A \to C_0 \otimes_{\Z\pi} A)}
			{\IM(d_2 \otimes \id_A \colon C_2 \otimes_{\Z \pi} A \to C_1 \otimes_{\Z\pi} A)}
	\]
	where $C_2 \xrightarrow[]{d_2} C_1 \xrightarrow[]{d_1} C_0 \to \Z \to 0$
	is a choice of free $\Z \pi$ resolution for the trivial $\Z \pi$-module $\Z$.
\end{convention}

\begin{convention}
	We adopt the following notation
	convention for maps $g \colon A \rightarrow B$
	between $\Z \pi$ modules, which can also occur in various combinations:
	\begin{itemize}
		\item A subscript $_*$ as in $g_* \colon \Gamma(A) \rightarrow \Gamma(B)$ denotes the
		induced map between $\Gamma$-groups.
		\item A hat $\widehat{\phantom{g}}$ as in
		$\widehat{g} \colon \widehat{H}_{i}(\pi, A) \rightarrow \widehat{H}_{i}(\pi, B)$
		denotes the map on Tate homology.
	\end{itemize}
\end{convention}

For a $\Z \pi$-module $A$,
let $[A]_{s}$  denote the equivalence class of $A$ up to \emph{stable isomorphism},
i.e.\ up to the relation where $A \sim_{s} B$ for a $\Z \pi$-module $B$
if there exists $i, j \ge 0$ for which $A \oplus \Z \pi^i \cong B \oplus \Z \pi^j$.
For later purposes, it will often be convenient to view this
as the set $[A]_{s} = \{ B : A \sim_{s} B \}$.

We conclude this section with the following observation.
\begin{lemma}[{\cite[Lemma~4.2]{KPR20}}]
	\label{lem:gamma-stable}
	Let $A$ be a $\Z \pi$-lattice.
	Then $\wh H_0(\pi;\Gamma(A))$ only depends on the stable isomorphism class $[A]_{s}$,
	i.e.\ if $A \sim_{s} B$ for a $\Z \pi$-module $B$,
	then there is an isomorphism of abelian groups $\wh H_0(\pi;\Gamma(A)) \cong \wh H_0(\pi;\Gamma(B))$.
\end{lemma}

In particular, in order to determine $\wh H_0(\pi;\Gamma(A))$ for a $\Z \pi$-module $A$, it suffices to consider $\wh H_0(\pi;\Gamma(B))$ for any $B$ inside the stable class $[A]_{s}$.

\section{Syzygies of finite groups}
\label{sec:syz}

In this section, we will recall the basic theory of syzygies of finite groups.
This offers an alternative perspective to some of the results which were discussed in \cite{KPR20}.

Recall that, for a finitely presented group $\pi$, a $\Z \pi$-module $A$  and $n \ge 1$, the \textit{$n$th stable syzygy} $\Omega_n(A)$, which we also write as $\Omega_n^{\pi}(A)$, is defined as the set of $\Z \pi$-modules $B$ for which there exists an exact sequence
\[ 0 \to B \to F_{n-1} \to \cdots \to F_0 \to A \to 0\]
where the $F_i$ are free $\Z \pi$-modules. 

The following was first shown by Swan in \cite[Corollaries 1.1, 2.1]{Sw60}. For a more recent reference, and a different proof, see \cite[Theorem 30.1]{Jo03-book}.

\begin{lemma}
	Let $\pi$ be a finite group, let $A$ be a $\Z \pi$-lattice and let $n \ge 2$.
	Then $\Omega_n(A) = [B]_{s}$ for any $B \in \Omega_n(A)$,
	i.e.\ if $B \in \Omega_n(A)$, then $B' \in \Omega_n(A)$
	if and only if $B$ and $B'$ are stably isomorphic.
\end{lemma}

The following is also immediate by noting that the exact sequence for $A$ and $B$ defined above is split when restricted to the underlying abelian groups provided $A$ is torsion free.

\begin{lemma}
Let $\pi$ be a finite group, let $A$ be a $\Z \pi$-lattice and let $n \ge 1$. If $B \in \Omega_n(\Z)$, then $B$ is a $\Z \pi$-lattice.
\end{lemma}

It is often useful to take the perspective (see \cite[Preface]{Jo12})
that the Syzygies $\Omega_n(A)$ are, in some sense,
the $n$th derivative of the module $A$.
This is already mentioned by R.\ H.\ Fox in his definition
of Fox derivative in 1960 \cite{Fo60}.
We will now recall this definition for use in the following section. 

\begin{definition}[Fox derivative]
If $F$ is a free group with generators $g_i$, then the \textit{Fox derivative} with respect to $g_i$
is the $\Z$-module homomorphism
\[ \partial_{g_i} : \Z F \to \Z F\]
which is defined by the requirements that $\partial_{g_i} (g_j) = \delta_{ij}$ where $\delta_{ij}$ is the Kronecker delta, $\partial_{g_i}(1)=0$, and the product rule $\partial_{g_i}(xy) = \partial_{g_i}(x) + x\partial_{g_i}(y)$ for $x, y \in F$.
If $\phi \colon F \twoheadrightarrow \pi$ is a surjection of groups, then we can view $\partial_{g_i}$ as a map $\partial_{g_i} : \Z F \to \Z \pi$ by post composition with $\phi$. In particular, $\partial_{g_i}$ maps words in the generators of $\pi$ to $\Z\pi$.
\end{definition}

The main result on Fox derivatives that concerns us is as follows. A detailed account can be found, for example, in \cite[Section 1.2]{MR17}.

\begin{prop} \label{prop:fox}
	Let $\mathcal{P} = \langle x_1, \cdots, x_n \mid r_1, \cdots, r_m \rangle$
	be a group presentation with corresponding presentation complex $X_{\mathcal{P}}$,
	and $\phi \colon F = \langle x_1, \ldots, x_n \rangle \twoheadrightarrow \pi$
	the corresponding surjection.
	Then the cellular chain complex of $\wt X_{\mathcal{P}}$ is given by
\[ 
	C_*(\wt X_{\mathcal{P}}) \colon
	\quad \underbrace{C_2(\wt X_{\mathcal{P}})}_{\displaystyle = \oplus_{i=1}^m \Z \pi \langle r_i \rangle} 
	\xrightarrow[]{\; d_2 \;}
	\underbrace{C_1(\wt X_{\mathcal{P}})}_{\displaystyle = \oplus_{i=1}^n \Z \pi \langle x_i \rangle} 
	\xrightarrow[]{\; d_1 \;} 
	\underbrace{C_0(\wt X_{\mathcal{P}})}_{\displaystyle = \Z \pi \langle 1 \rangle}
\]
	where the maps (of left $\Z \pi$-modules)
	are given on the basis vectors as
	$d_2(r_i) = \sum_{j=1}^n \phi( \partial_{x_j}(r_i) ) \cdot x_j$
	and $d_1(x_j)=\phi(x_{j}) - 1$ for all $i,j$.
\end{prop}

If $\mathcal{P}$ is a presentation for $\pi$, then $\ker(d_2) \in \Omega_3^{\pi}(\Z)$. Hence, in order to find an explicit parametrisation of $\Omega_3^{\pi}(\Z)$, it remains to compute $\ker(d_2)$ for some presentation $\mathcal{P}$. 

\begin{remark} Whilst this method works to obtain a parametrisation for $\Omega_3(\Z)$, it is currently not known whether or not this method can always be used to find a representative whose abelian group has minimal rank. For example, it was noted by the second named author in \cite{Ni19}, that there is a family of groups $\pi = P''_{48 \cdot n}$ for $n \ge 3$ odd with 4-periodic cohomology over which there exists $J \in \Omega_3(\Z)$ with $\rank_{\Z}(J) = |\pi|-1$ but for which every known presentation $\mathcal{P} = \langle x_1, \cdots, x_n \mid r_1, \cdots, r_m \rangle$ has $m-n \ge 1$ and so has $\rank_{\Z}(\ker(d_2)) \ge 2|\pi|-1$.
\end{remark}

For a $\Z \pi$-module $A$, define the \textit{dual} $A^* = \Hom_{\Z}(A,\Z)$ which has left $\Z \pi$ action given by sending $\varphi \mapsto g \cdot \varphi$ where $(g \cdot \varphi)(x) = \varphi(g^{-1} \cdot x)$ for $x \in A$. By, for example, \cite[Lemma 1.5]{Ni20-I} this coincides with the usual dual of $\Z \pi$-modules $\Hom_{\Z \pi}(A,\Z\pi)$.
From our definition it is clear that for $\pi$ finite,
if $A$ is a $\Z \pi$-lattice, then $A^*$ is also a $\Z \pi$-lattice.

The following was proven by Hambleton--Kreck \cite[Proposition 2.4]{HK}.

\begin{prop}
	\label{prop:pi_2-SES}
	Let $X$ be a finite oriented Poincar\'{e} $4$-complex $X$
	with finite fundamental group $\pi$.
	Then there exists $J \in \Omega_3(\Z)$, an integer $r \ge 0$ and an exact sequence 
	\[ 
		0 \to J \to \pi_2(X) \oplus \Z\pi^r \to J^* \to 0.
	\]
\end{prop}

By the discussion above, we know that $J$ and $J^*$ are necessarily $\Z \pi$-lattices.
By combining Lemmas \ref{lem:bauer} and \ref{lem:gamma-stable}, it is straightforward to show the following. See \cite[Corollary~4.5]{KPR20} for a detailed proof.

\begin{corollary}
	\label{cor:torsionfree}
	If $\Tors(\Z\otimes_{\Z\pi} \Gamma(J)) = 0$ and
	$\Tors(\Z\otimes_{\Z\pi} \Gamma(J^*)) = 0$ for some 
	$J \in \Omega_3(\Z)$,
	then $\Tors(\Z\otimes_{\Z\pi} \Gamma(\pi_2(X))) = 0$.
\end{corollary}

If $A$ is a $\Z \pi$-lattice then, by \Cref{lemma:torsion-gamma-equals-tate},
we have $\Tors(\Z \otimes_{\Z\pi} A) \cong \wh H_0(\pi;A)$.
It is well known (see, for example, \cite[III (10)]{brown}),
that this vanishes if and only if it vanishes over each Sylow $p$-subgroup $\pi_p$. 

Using this, Bauer made the following observation \cite[p.\ 5]{Ba88} in the case $n=3$ (see \cite[Section 6]{KPR20} for additional details). By examining the argument, it is not difficult to see that this extends to all $n \ge 1$ odd.

\begin{lemma} \label{lemma:sylow-2}
	Let $\pi$ be a finite group with Sylow $2$-subgroup $\pi'$.
	For $n \ge 1$ odd, let $J \in \Omega^{\pi}_n(\Z)$,
	and let $J' = \Res^{\pi}_{\pi'}(J) \in \Omega_n^{\pi'}(\Z)$ denote its restriction to $\Z\pi'$.
	If $\Tors(\Z \otimes_{\Z \pi'} \Gamma(J')) = 0$,
	then $\Tors(\Z \otimes_{\Z \pi} \Gamma(J)) = 0$.
	Similarly, if $\Tors(\Z \otimes_{\Z \pi'} \Gamma((J')^*)) = 0$,
	then $\Tors(\Z \otimes_{\Z \pi} \Gamma(J^*)) = 0$.
\end{lemma}

We conclude this section with an overview of the proof of \cref{thm:main}. As we mentioned in the introduction, it is a consequence of \cite[Theorem~1.1]{HK} and \cite{teichnerthesis} (see~\cite[Corollary~1.5]{KT}) that, in order to prove \cref{thm:main}, it suffices to prove that $\Tors(\Z\otimes_{\Z\pi} \Gamma(\pi_2(X))) = 0$ when $\pi$ is a finite group whose Sylow 2-subgroup is dihedral. 
By Corollary \ref{cor:torsionfree} and Lemma \ref{lemma:sylow-2}, it suffices to prove that $\Tors(\Z \otimes_{\Z \pi} \Gamma(J)) = 0$ and $\Tors(\Z \otimes_{\Z \pi} \Gamma(J^*)) = 0$ where $\pi$ is the dihedral group of order $2^n$ for $n \ge 1$. These results will be shown in \cref{thm:ker=0} and \cref{thm:coker=0} respectively.

\section{An explicit parametrisation for $\Omega_3(\Z)$ over dihedral groups}
\label{sec:ker1}

The aim of this section we will be to obtain an explicit parametrisation for $\Omega_3(\Z)$ in the case where $\pi = D_{2n}$ is the dihedral group of order $2n$ where $n$ is even. 
Note that, if $n$ is odd, then $D_{2n}$ has 4-periodic cohomology and so is dealt with by the results of Hambleton--Kreck \cite{HK}. In fact, it is possible to parametrise all the syzygies $\Omega_m^{D_{2n}}(\Z)$ for $m \ge 1$ in this case \cite{Jo16}.

Using the presentation
\[
	\mathcal{P} = \langle x, y \mid x^n y^{-2}, xyxy^{-1}, y^2 \rangle
\]
for $D_{2n}$,
we obtain the following partial free resolution of $\Z$
using \Cref{prop:fox}
\begin{equation}
	\label{eqn:resolution_dihedral_group}
	C_*(\mathcal{P}) \colon \quad
	0 \to \ker(d_2) 
	\to \Z \pi^3
	\xrightarrow[d_2]{\cdot 
		\left(
		\begin{smallmatrix}
			 N_x & -(1+y) \\ 
			 1+xy & x-1 \\ 
			 0 & y+1 
		\end{smallmatrix}
		\right)}
	\Z \pi^2
	\xrightarrow[d_1]{\cdot \left(\begin{smallmatrix} x-1 \\ y-1\end{smallmatrix}\right)}
	\Z \pi
	\xrightarrow[]{\varepsilon} \Z \to 0
\end{equation}
where $N_x = 1+x+\cdots+x^{n-1}$ and $\varepsilon$ is the augmentation map.
Here the matrices describing the 
left $\Z \pi$-linear differentials $d_1, d_2$ multiply from the right, 
with the elements of the free $\Z \pi$-modules written as row vectors.
In particular, the composition corresponds to the
matrix product $(d_{1} \circ d_{2})(v) = v \cdot d_{2} \cdot d_1$.
Let $N = \sum_{g \in \pi} g$ denote the group norm. Then:

\begin{lemma}
	\label{lem:quaternion_group}
	The following sequence is exact:
	\[ 
	0 \to 
	\Z \xrightarrow[]{N} 
	\Z \pi \xrightarrow[]{
		\cdot
		\left( \begin{smallmatrix}
			x-1 & 1-xy 
		\end{smallmatrix} \right)}
	\Z \pi^2 \xrightarrow[]{
		\cdot 
		\left( \begin{smallmatrix} 
			N_x & -(1+y) \\ 
			1+xy & x-1
		\end{smallmatrix} \right)
	}
	\Z \pi^2
	\]
\end{lemma}

\begin{proof}
	This can be checked directly.
	However, let us give a shorter proof imitating \cite[Lemma 2.4]{HK93}.
	Consider the 4-periodic resolution
	\[  
		\Z Q_{4n} \xrightarrow[]{N_{Q_{4n}}} 
		\Z Q_{4n} \xrightarrow[]{
			\cdot
			\left( \begin{smallmatrix}
				x-1 & 1-xy 
			\end{smallmatrix} \right)}
		\Z Q_{4n}^2 \xrightarrow[]{
			\cdot 
			\left( \begin{smallmatrix} 
				N_x & -(1+y) \\ 
				1+xy & x-1
			\end{smallmatrix} \right)
		}
		\Z Q_{4n}^2 \xrightarrow{
			\cdot
			\left( \begin{smallmatrix}
				x-1 \\
				y-1 
			\end{smallmatrix} \right)}
		\Z Q_{4n}
	\]
	of $\Z$ over the generalized quaternion group
	$Q_{4n}$ from \cite[p.\ 253]{cartan1956homological}.
	The beginning of this resolution corresponds to the presentation
	$\langle x, y \mid x^n y^{-2}, xyxy^{-1} \rangle$ of $Q_{4n}$.
	
	Apply the functor $- \otimes_{\Z[\langle y^2\rangle]} \Z$,
	where $\langle y^2 \rangle \subset Q_{4n}$ is the cyclic group
	$C_{2}$ with two elements.
	Since $\Tor_3^{ \Z [ \langle y^2 \rangle ]}(\Z, \Z) \cong H_3(C_2;\Z) \cong \Z/2$
	it does not remain exact at the third term,
	but as
	$\Tor_2^{\Z [ \langle y^2 \rangle ]}(\Z, \Z) \cong H_2(C_2;\Z) = 0$,
	we conclude that
	\[	
		\Z\pi \xrightarrow[]{
		\cdot
		\left(
		\begin{smallmatrix}
			x-1 & 1-xy 
		\end{smallmatrix}
		\right)}
		\Z\pi^2 \xrightarrow[]{
		\cdot 
		\left(
		\begin{smallmatrix} 
			N_x & -(1+y) \\ 
			1+xy & x-1
		\end{smallmatrix}
		\right)}
		\Z\pi^2 
	\]
	is still exact. 
	Note that the kernel of
	$\Z\pi \xrightarrow[]{
			\cdot
			\left( \begin{smallmatrix}
			x-1 & 1-xy 
			\end{smallmatrix} \right)}
	\Z\pi^2$
	is the set of fixed points under the $\pi$-action and so is the image of the norm map.
	This implies the lemma.
\end{proof}

\begin{lemma}
	\label{lem:kernel_decomposition}
	Let $f=(f_A,f_B) \colon A \oplus B \to C$
	be a map between abelian groups.
	Then there is an exact sequence
	\[
	\begin{tikzcd}
		0 \ar[r] & 
		\Ker(f_A) \ar[r,"i"] & 
		\Ker(f) \ar[r,"j"] & 
		\Ker(q \circ f_B \colon B \to C/\IM(f_A)) \ar[r] & 0
	\end{tikzcd}
	\]
	where $i \colon a \mapsto (a,0)$,
	$j \colon (a,b) \mapsto b$ and
	$q \colon C \mapsto C/\IM(f_A)$ is the quotient map.
\end{lemma}

\begin{proof}
	It is easy to see that $i$ is injective and that $\IM(i) = \Ker(j)$.
	To show that $j$ is surjective, let $b \in \Ker(q \circ f_B \colon B \to C/\IM(f_A))$.
	Then $f_B(b) \in \IM(f_A)$ so there exists $a \in A$ such that $f_A(a)=f_B(b)$ and so $j(-a,b) = b$.
\end{proof}

\begin{definition}
	We denote the \emph{augmentation ideal} by
	$I = I\pi = \ker (\Z \pi \xrightarrow{\varepsilon} \Z)$
	and the \emph{ideal generated by $I$ and $2$} by
	$(I, 2) = \ker (\Z \pi \xrightarrow{\varepsilon} \Z / 2)$.
\end{definition}

\begin{remark}	
	\label{rem:dualI}
	Dualising the exact sequence
	\[0\to I\to \Z\pi\xrightarrow{\varepsilon}\Z\to 0 \]
	we obtain the exact sequence
	\[0\to \Z\xrightarrow{N}\Z\pi\to I^*\to 0.\]
	In particular, the dual of $I$ is isomorphic to $\Z\pi/N$.
\end{remark}

\begin{prop}
	\label{prop:kernel_exact_sequence}
	With respect to the inclusion $\Ker(d_2) \subseteq \Z \pi^3$,
	there is an exact sequence:
	\[
		0 \to 
		\Z \pi / N \xrightarrow[i]{\cdot \left(\begin{smallmatrix} x-1 & 1-xy & 0 \end{smallmatrix}\right)}
		\Ker(d_2) \xrightarrow[j]{\cdot \left(\begin{smallmatrix} 0 \\ 0 \\ 1 \end{smallmatrix}\right)} 
		(I,2) \to 0.
	\]
	Furthermore, we have $j(1+y,-N_x,2)=2$, $j(x-1,0,x-1)=x-1$ and $j(0,0,y-1)=y-1$
	which gives lifts of the $\Z \pi$-module generators for $(I,2)$.
\end{prop}

\begin{proof}
	Follows by applying the decomposition in \Cref{lem:kernel_decomposition}
	to the $d_2$ differential in the resolution
	(\ref{eqn:resolution_dihedral_group}) for the dihedral group.
	Here $f_A = \cdot 
		\left( \begin{smallmatrix}
			N_x & -(1+y) \\ 
			1+xy & x-1 
		\end{smallmatrix} \right)$
	corresponds to the first two rows of the matrix, and
	$f_B = \cdot \left(\begin{smallmatrix} 0 & y + 1 \end{smallmatrix}\right)$
	to the bottom row.
	Now use \Cref{lem:quaternion_group} to identify
	$\ker f_{A}$ with $\Z \pi / N$.
	
	To identify $\ker(q \circ f_{B} \colon \Z \pi \rightarrow \Z \pi^{2} / \IM f_{A})$,
	consider again the resolution
	\[  
		\Z Q_{4n} \xrightarrow[]{N_{Q_{4n}}} 
		\Z Q_{4n} \xrightarrow[]{
			\cdot
			\left( \begin{smallmatrix}
			x-1 & 1-xy 
			\end{smallmatrix} \right)}
		\Z Q_{4n}^2 \xrightarrow[]{
			\cdot 
			\left( \begin{smallmatrix} 
			N_x & -(1+y) \\ 
			1+xy & x-1
			\end{smallmatrix} \right)
		}
		\Z Q_{4n}^2 \xrightarrow{
			\cdot
			\left( \begin{smallmatrix}
			x-1 \\
			y-1 
			\end{smallmatrix} \right)}
		\Z Q_{4n}
	\]
	of $\Z$ over the generalized quaternion group
	$Q_{4n}$ from the proof of \cref{lem:quaternion_group}.
	Since $\Tor_1^{\Z [ \langle y^2 \rangle]}(\Z, \Z) \cong H_1(C_2;\Z)\cong \Z/2$,
	the sequence
	\begin{equation}
	\label{homologyZ2}
			\Z \pi^2 \xrightarrow[]{f_A=
			\cdot 
			\left( \begin{smallmatrix} 
			N_x & -(1+y) \\ 
			1+xy & x-1
			\end{smallmatrix} \right)
		}
		\Z \pi^2 \xrightarrow{
			\cdot
			\left( \begin{smallmatrix}
			x-1 \\
			y-1 
			\end{smallmatrix} \right)}
		\Z\pi
	\end{equation}	
	has homology $\Z/2$.
	As $f_B=\cdot \left(\begin{smallmatrix} 0 & y + 1 \end{smallmatrix}\right)$
	composed with
	$\cdot \left( \begin{smallmatrix}
	x-1 \\
	y-1 
	\end{smallmatrix} \right)$ 
	is trivial, the map
	$q \circ f_{B} \colon \Z \pi \rightarrow \Z \pi^{2} / \IM f_{A}$ factors through
	\[
	\Z/2=\ker\left(	\Z \pi^{2} / \IM f_{A} \xrightarrow{
		\cdot
		\left( \begin{smallmatrix}
		x-1 \\
		y-1 
		\end{smallmatrix} \right)}
	\Z\pi\right).
	\]
	To see that $\ker(q \circ f_{B} \colon \Z \pi \rightarrow \Z \pi^{2} / \IM f_{A})\cong (I,2)$
	it remains to show that $(0,y+1)$ is non-trivial in $\Z \pi^{2} / \IM f_{A}$.
	Assume that $(0,y+1)$ is in the image of $f_A$, then the exactness of \eqref{eqn:resolution_dihedral_group}
	implies that \eqref{homologyZ2} is exact. But \eqref{homologyZ2} has homology $\Z/2$ as mentioned above.
\end{proof}

\begin{remark}
	It will also be useful to note that $j(0,xy-1,xy-1) = xy-1$.
	The following equalities in $\Z \pi$ will be used without
	comment in our calculations:
	\begin{itemize}
		\item $(1 + x^{k} y)(1 - x^{k} y) = 0 = (1 - x^{k} y)(1 + x^{k} y)$
		\item $\overline{x^{k} y} = x^{k} y$, 
		$\overline{N_x} = N_x$, 
		$\overline{1 \pm xy} = 1 \pm xy$, 
		$\overline{1+y} = 1+y$
		\item $(1 - xy)(x + y) = 0$
		\item $xy-1 = (x-1)y+(y-1)$
	\end{itemize}
	Here $\overline{\phantom{x}}$ is the usual involution on the group ring $\Z \pi$
	induced by sending $g \mapsto g^{-1}$ for $g \in \pi$.
\end{remark}

\section{Computing \texorpdfstring{$\widehat{H}_0(\pi;\Gamma(\Ker(d_2)))$}{the kernel}}
\label{sec:ker2}

The aim of this section will be the following theorem,
whose proof appears on page~\pageref{proof:ker=0}.
\begin{thm}
	\label{thm:ker=0}
	If $\pi$ is a dihedral group of order $2n$ for $n$ even,
	then $\widehat{H}_0(\pi;\Gamma(\Ker(d_2)))=0$.	
\end{thm}

\begin{remark}
	Computer assisted calculations verifying
	the vanishing of $\widehat{H}_0(\pi;\Gamma(\Ker(d_2)))$
	where $\pi = D_{2n}$ and $n \le 24$ can be found at \cite{algorithm}.
\end{remark}

Let $D = \Gamma(\Ker(d_2)) / \Gamma(\Z \pi/N)$,
i.e.\ so that  there is an exact sequence
\[ 
	0 \to
	\Gamma(\Z \pi/N) \xrightarrow[]{i_*}
	\Gamma(\Ker(d_2)) \xrightarrow[]{q} D
	\to 0
\]
where $q$ is the quotient map. By \cref{lem:bauer} applied to the
decomposition of $\ker(d_2)$ in \Cref{prop:kernel_exact_sequence} there is an exact sequence
\[ 
	0 \to (\Z \pi/N) \otimes_{\Z} (I,2)
	\xrightarrow[]{f} D
	\xrightarrow[]{j_*} \Gamma((I,2))
	\to 0.
\]
By the work done previously, the map $f$ is given by
\begin{align*}
	f \colon 
	(\Z \pi/N) \otimes_{\Z} (I,2) & \to  
	D = \Gamma(\Ker(d_2)) / \Gamma(\Z \pi/N) \\
	1 \otimes 2  &\mapsto   [(x-1,1-xy,0) \odot (1+y,-N_x,2)]  \\
	1 \otimes (y-1)  &\mapsto [(x-1,1-xy,0) \odot (0,0,y-1)] \\
	1 \otimes (xy-1)  & \mapsto  [(x-1,1-xy,0) \odot (0,xy-1,xy-1)]
\end{align*}
Here we decided to define the map $f$ using lifts of the elements
$2, y-1, xy-1$ as opposed to $2, y-1, x-1$, since the following calculations
will be easier with respect to the generating set
$x, xy$ consisting of order $2$ elements of $D_{2n}$. 

Now consider the long exact sequence on Tate homology coming from the first exact sequence.
By \cite[Theorem~2.1]{HK} and \cref{rem:dualI}, 
we have that $\widehat{H}_0(\pi; \Gamma(\Z \pi/N)) = 0$ and so we have:
\[ 
	\ldots \to
	0 \to \widehat{H}_0(\pi;\Gamma(\Ker(d_2)))
	\xrightarrow[]{\widehat{q}} \widehat{H}_0(\pi;D)
	\xrightarrow[]{\partial} \widehat{H}_{-1}(\pi;\Gamma(\Z \pi/N))
	\to \ldots
\]
where $\partial$ denotes the boundary map in Tate homology.

We will now prepare a sequence of lemmas,
which will then lead to a proof 
of the following \Cref{prop:H_0(D)} on page~\pageref{proof:H_0(D)}.
From now on, let
\[
	\sigma
	=(1+yx)\sum_{i=1}^{n/2}x^{2i} 
	= (1+yx)(x^2 + x^4 + \ldots + x^{n})
	= (1+yx)(1 + x^2 + x^4 + \ldots + x^{n-2}).
\]
\begin{prop}
	\label{prop:H_0(D)}
	There is an isomorphism of abelian groups
	\[
		\widehat{H}_{0}(\pi;D) 
		\cong 
		\Z/2 \, \langle \alpha_1 \rangle \oplus \Z/2 \, \langle \alpha_2 \rangle
	\]
	where the images in $\widehat{H}_{-1}(\pi;\Gamma(\Z \pi/N))$
	of the generators $\alpha_1, \alpha_2$
	under the boundary map are
	\begin{align*}
		\partial(\alpha_1) & = 2 \cdot (N_x \otimes N_x) \text{ and} \\
		\partial(\alpha_2) 
		& = n \cdot (N_x \otimes N_x) + 2 \cdot (\sigma \otimes \sigma).
	\end{align*}
\end{prop}
Remember that we use the notation
from \Cref{convention} to denote the equivalence classes of the elements
$2 \cdot (N_x \otimes N_x), n \cdot (N_x \otimes N_x) + 2 \cdot (\sigma \otimes \sigma)
\in \wh H_{-1}(\pi;\Gamma(\Z \pi/N))$
that live in the cokernel of the norm map.

We begin by noting that we have the following long exact sequence on Tate homology:
\[
	\ldots \to
	\widehat{H}_0(\pi;(\Z\pi/N) \otimes_\Z (I,2)) \xrightarrow[]{\widehat{f}}
	\widehat{H}_0(\pi;D) \xrightarrow[]{\widehat{j_*}}
	\widehat{H}_0(\pi;\Gamma((I,2)))
	\to \ldots
\]

\begin{lemma}
	\label{lemma:dim-shift}
	For every finite group $G$ of even order, there is an isomorphism of abelian groups
	\[ 
		\widehat{H}_0(G;(\Z G/N)\otimes_{\Z} (I,2))
		\cong
		\Z/2 \, \langle 1 \otimes N \rangle.
	\]
\end{lemma}

\begin{proof}
	First consider the short exact sequence
	$0\to (I,2)\to \Z G\to \Z/2\to 0$.
	Since the order of $G$ is even, the norm map is trivial on $\Z/2$
	and hence $1\in \Z/2$ is non-trivial in $\wh H_0(G;\Z/2)$.
	In particular, $\wh H_0(G;\Z/2)\cong \Z/2$.
	By dimension shifting, i.e.\ using that the Tate homology of $\Z G$ vanishes,
	we get $\wh H_0(G;\Z/2)\cong \wh H_{-1}(G;(I,2))$.
	The pre-image $1\in\Z G$ maps to $N\in\Z G$ under the norm map
	and hence $N \in \wh H_{-1}(G;(I,2))$ is the non-trivial element.

	Now consider the short exact sequence
	\begin{equation}\label{seq:I2} 0 \to
	\Z \otimes_{\Z} (I,2) 
	\xrightarrow[]{N \otimes 1}
	\Z G \otimes_{\Z} (I,2)
	\xrightarrow[]{1 \otimes 1} (\Z G/ N) \otimes_{\Z} (I,2) \to 0\end{equation}
	where the middle term is free by \cite[Lemma 4.3]{KPR20}. 
	By dimension shifting, we have
	\[
		\wh H_0(G;(\Z G/N) \otimes_{\Z} (I,2))\cong \wh H_{-1}(G;(I,2))\cong \Z/2.
	\]
	Since $N$ is a fixed point under the $G$ action,
	the element $1\otimes N \in (\Z G/N) \otimes_\Z (I,2)$ maps to $0=N\otimes N$
	under the norm map. Hence it represents an element of $\wh H_0(G; (\Z G/N) \otimes_\Z (I,2))$.
	Under the boundary map induced from the sequence \eqref{seq:I2} it is mapped to
	$1\otimes N \in \wh H_{-1}(G;\Z\otimes_\Z (I,2))$ which is the non-trivial element
	by the previous calculation.
	This implies that $1\otimes N$  
	represents the non-trivial element in $\wh H_0(G;(\Z G/N) \otimes_\Z (I,2))$.
\end{proof}

\begin{lemma}
	\label{lemma:f_*=0}
	The map
	$\widehat{H}_0(\pi;(\Z \pi/N)\otimes_{\Z} (I,2))
	\xrightarrow{\widehat{f}}
	\wh H_0(\pi;D)$
	is trivial, i.e.\ $\widehat{f}(1 \otimes N) = 0$.
\end{lemma}

\begin{proof}
	A lift of $N\in (I,2)$ in $\ker d_2$ is given by $(N,-\frac{n}{2}N,N)$.
	Hence we have
	\[ 
		f(1\otimes N) = (x-1,1-xy,0) \odot (N,-\frac{n}{2}N,N).
	\]
	It is straightforward to verify that there are the 
	following three ways to decompose
	\[ 
		(N,-\frac{n}{2} N, N) 
		= N_x v_1
		= (1+xy)\sum_{i=1}^{n/2} x^{2i}v_2
		=  (1+y)\sum_{i=1}^{n/2} x^{2i}v_3,
	\] 
	where
	\begin{align*}
		v_1 &= (x+y-\frac{n}{2} (x-1), -yN_x-\frac{n}{2} (1-xy), x+y) \\
		v_2 &= (1+y,-N_x,1+y) \\
		v_3 &= (1+yx,-N_x,1+yx)
	\end{align*}
	are all in $\Ker(d_2)$.
	Note that we can also decompose the other factor in $\Ker(d_2)$ as
	\[
		(x-1,1-xy,0) = (x-1,0,x-1) + (0,1-xy,1-xy) + x(0,0,y-1).
	\]
	
	In the situation where $\Z \pi$ acts diagonally on a tensor
	product $L \otimes_{\Z} L$ of $\Z \pi$-modules,
	in the tensored down module
	$\Z \otimes_{\Z \pi} (L \otimes_{\Z} L) \cong L \otimes_{\Z \pi} L$ the
	relation $a \otimes (\lambda b) = (\overline{\lambda} a) \otimes b$
	holds, where $\lambda \in \Z \pi$ and $a,b \in L$.
	The elements $N_x,1+xy$ and $1+y$ are invariant under
	applying the involution $\overline{\phantom{x}}$,
	so we use this in $D \otimes_{\Z \pi} D$ to move them freely
	between the factors in the tensor products in the first equalities below.
	We have in 
	$\wh H_{0}(\pi; D) \cong \Tors \Z \otimes_{\Z\pi} ( \Gamma(\Ker(d_2)) / \Gamma(\Z \pi/N) )$
	that the following tensors all vanish: 
	\begin{align*} 
		(x-1,0,x-1) \otimes N_xv_1 & = N_x(x-1,0,x-1) \otimes v_1= 0 \\
	 	(0,1-xy,1-xy) \otimes  (1+xy)\sum x^{2i}v_2 & = (1+xy)(0,1-xy,1-xy) \otimes \sum x^{2i} v_2 = 0 \\
	 	x(0,0,y-1) \otimes  (1+y)\sum x^{2i}v_3
	 	&= (0,0,y-1) \otimes  x^{-1} (1+y) \sum x^{2i}v_3 \\
	 	&= (0,0,y-1) \otimes  (1+y) \sum x^{2i+1}v_3 \\
	 	&= (1+y)(0,0,y-1) \otimes  \sum x^{2i+1}v_3 = 0.
 	\end{align*}
	By adding together these expressions, we get
	$(x-1,1-xy,0) \odot (N,-\frac{n}{2} N, N) = 0$
	in $\wh H_{0}(\pi; D)$.
\end{proof}

We will now compute $\widehat{H}_0(\pi;\Gamma((I,2)))$. 
First note that there is an exact sequence
\[
	0 \to I \hookrightarrow (I,2) \xrightarrow{\varepsilon} 2\Z \to 0.
\]

\begin{lemma}
	\label{lem:gammaquotient}
	There is an isomorphism of $\Z \pi$-modules 
	\[
		\varphi \colon (I,2) \to { \Gamma((I,2)) }/{ \Gamma(I) }
	\]
	given by $2 \mapsto 2 \otimes 2$ and
	$g-1  \mapsto 2 \odot (g-1)$ for $g \in \pi$.
\end{lemma}

\begin{proof}
	A $\Z$ basis of $(I,2)$ is given by $2$ and all $g-1$ for $g\in\pi\setminus \{1\}$.
	By \cref{lem:gammafree} a $\Z$ basis for $\Gamma((I,2))$ is thus given by 
	\[
		\{ (g-1) \otimes (g-1), \, 2 \otimes 2, \, 2 \odot  (g-1), \, (g-1) \odot (h-1)
		\mid
		g, h \in \pi \setminus \{ 1\}, g\ne h \}.
	\]
	Doing the same consideration for $I$, we see that a $\Z$ basis for $\Gamma((I,2))/\Gamma(I)$
	is given by
	\[
		\{ [2 \otimes 2], \, [2 \odot (g-1)]
		\mid
		g \in \pi \setminus \{1 \} \}.
	\]
	Thus the map $\varphi$ is a bijection of $\Z$-modules.
	
	It remains to show that $\varphi$ is $\Z\pi$-linear. Let $g\in\pi$ be given. Then
	\begin{align*}
		(g-1)\cdot(2 \otimes 2) - 2(2 \odot (g-1))
		& = g \cdot (2 \otimes 2) - 2 \otimes 2 - (2 \otimes 2(g-1) + 2(g-1) \otimes 2) \\
		& = (2g \otimes 2g) - 2 \otimes 2 - 2 \otimes 2g + 2 \otimes 2 - 2g \otimes 2 + 2 \otimes 2 \\
		& = (2g - 2) \otimes (2g - 2)
		= 4(g-1) \otimes (g-1) \in \Gamma(I)
	\end{align*}
	and so $[(g-1)(2 \otimes 2)] = [2( 2 \odot (g-1))] \in \Gamma((I,2))/\Gamma(I)$. Hence 
	\[\varphi(2g)=\varphi(2(g-1))+\varphi(2)=[2(2 \odot (g-1))]+[2\otimes 2]=[(g-1)(2 \otimes 2)]+[2\otimes 2]=g[2\otimes 2]=g\varphi(2).\]
	Similarly, for $g,h\in\pi$ we have $2(g-1)\odot g(h-1)\in \Gamma(I)$ and hence \[[2\odot g(h-1)]=[2g\odot g(h-1)]=g[2\odot (h-1)]=g\varphi(h-1).\] Thus
	\[\varphi(g(h-1))=\varphi(gh-1)-\varphi(g-1)=[2\odot gh-1]-[2\odot g-1]=[2\odot g(h-1)]=g\varphi(h-1).\qedhere\]
\end{proof}

\begin{lemma}
	\label{lem:H0I2}
	For every finite group $G$ there is an isomorphism of abelian groups
	\[
		G^{\textup{ab}}\otimes_\Z \Z/2\cong \wh H_0(G;(I,2))
	\]
	sending $g\otimes 1$ to $g-1$.
\end{lemma}
\begin{proof}
	Consider the exact sequence
	\[ 
		0 \to (I,2) \hookrightarrow \Z G \xrightarrow[]{\varepsilon} \Z/2 \to 0.
	\]
	The boundary map $\wh H_1(G;\Z/2)\xrightarrow{\partial} \wh H_0(G;(I,2))$ is an isomorphism since $\wh H_*(G;\Z G)=0$. Thus
	\[\wh H_0(G;(I,2))\cong \wh H_1(G;\Z/2)\cong H_1(G;\Z/2)\cong G^{\mathrm{ab}}\otimes_\Z \Z/2.\]
	We now compute the boundary map explicitly,
	adopting \Cref{convention} for the notation in the diagram:
	\[
	\begin{tikzcd}[column sep = 1cm]
		0 \ar[r] & C_1 \otimes_{\Z G} (I,2) \ar[d,"d_1 \otimes \id"] \ar[r] & 
		C_1 \otimes_{\Z G} \Z G \ar[d,"d_1 \otimes \id"] \ar[r,"\id \otimes \varepsilon"] & 
		C_1 \otimes_{\Z G} \Z/2 \ar[d,"d_1 \otimes \id"] \ar[r] & 0 \\
		0 \ar[r] & C_{0} \otimes_{\Z G} (I,2) \ar[r] & 
		C_0 \otimes_{\Z G} \Z G \ar[r,"\varepsilon"] & C_{0} \otimes_{\Z G} \Z/2 \ar[r] & 0
	\end{tikzcd}
	\]	
	For $g\in G$ let $c_g\in C_1$ be a preimage of $(g-1)\in C_0\cong \Z G$ under $d_1$.
	Then $c_g\otimes 1\in C_1\otimes_{\Z G}\Z/2$ represents
	$g\otimes 1\in G^{\textup{ab}}\otimes \Z/2$.
	Under the boundary map
	\[
		G^{\mathrm{ab}}\otimes_\Z\Z/2\cong \wh H_1(G;\Z/2)\xrightarrow{\partial} \wh H_0(G;(I,2)),
	\]
	$g\otimes 1$ is send to 
	$d_1(c_g) \otimes 1 = 1 \otimes (g-1) \in C_0\otimes_{\Z G}(I,2)$.
\end{proof}

\begin{lemma}
	\label{lemma:H_0(G(I2))}
	There is an isomorphism of abelian groups
	\[ 
		\widehat{H}_0(\pi;\Gamma((I,2))) 
		\cong 
		\Z/2 \, \langle \alpha_{xy} \rangle \oplus \Z/2 \, \langle \alpha_y \rangle
	\]
	where for $g \in \pi$ which satisfy $g^2 =1$ we introduce the notation
	\[
		\alpha_g = 2 \odot (g-1) + 2 (g-1) \otimes (g-1) \in \widehat{H}_0(\pi;\Gamma((I,2))).
	\]
\end{lemma}

\begin{remark}
	We can also write this as
	$\alpha_g = 2 \otimes (g-1) - (g-1) \otimes 2$.
	Observe that for $g$ of order 2, we have
	$g(g-1) = -(g-1)$ and $(g-1)^2 = -2(g-1)$ in $\Z \pi$.
\end{remark}

\begin{proof}
	We first show that the elements $\alpha_g$ are torsion in $\Z\otimes_{\Z \pi} \Gamma((I,2))$ and hence represent elements in $\widehat{H}_0(\pi;\Gamma((I,2)))$. Using that $g^2=1$, in $\Gamma((I,2))$ we have
	\begin{align*}
	(1+g)\alpha_g
	& = (1+g)(2\odot(g-1) + 2(g-1)\otimes(g-1)) \\
	& = 2\odot (g-1)-2g\odot(g-1) + 4(g-1)\otimes(g-1) \\
	& = -2(g-1)\odot (g-1)+ 4(g-1)\otimes(g-1) \\
	& = -4(g-1)\otimes(g-1)+ 4(g-1)\otimes(g-1) \\
	& = 0
	\end{align*}
	Hence in $\Z\otimes_{\Z \pi}\Gamma((I,2))$ the elements $\alpha_g$ are 2-torsion,
	since multiplication by $2$ and by $(1+g)$ are equivalent under the trivial action on the first factor.
		
	Now consider the short exact sequence
	\[
		0 \to \Gamma(I) \to \Gamma((I,2)) \xrightarrow[]{\psi} (I,2) \to 0,
	\]
	using the isomorphism $\Gamma((I,2))/\Gamma(I)\cong (I,2)$ from \cref{lem:gammaquotient}.
	Under this isomorphism, the elements $\alpha_g$ map to $g-1$.
	Hence by \cref{lem:H0I2}, the map
	\[
		\wh H_0(\pi;\Gamma((I,2)))\xrightarrow{\widehat{\psi}} \wh H_0(\pi;(I,2))
	\]
	is surjective. Here we use that the dihedral group $\pi$ of order $2n$ for $n$ even is generated by $xy$ and $y$ which are both 2-torsion,
	and thus generate the abelianization
	$\pi^{\textup{ab}} \cong \Z/2 \oplus \Z/2$.
	By \cite[Theorem~2.1]{HK}, we have that $\widehat{H}_0(\pi;\Gamma(I)) = 0 $ and so the map
	$\widehat{\psi} \colon \wh H_0(\pi;\Gamma((I,2))) \rightarrow \wh H_0(\pi;(I,2))$
	is an isomorphism by the long exact sequence on Tate homology. 
\end{proof}

Recall that we defined the
element $\sigma=(1+yx)\sum_{i=1}^{n/2}x^{2i}$.
We note the following properties that we will use in our calculations:
$ x \sigma = x^{-1} \sigma = y \sigma = \sigma x = \sigma x^{-1} = \sigma y $.
The following lemma concerns the images of the maps
\[
\widehat{H}_0(\pi;D) \xrightarrow[]{\widehat{j_*}}
\widehat{H}_0(\pi;\Gamma((I,2)))
\text{ and } 
\Gamma(\Z \pi/N) \xrightarrow[]{i_*}
\Gamma(\Ker(d_2)) \xrightarrow[]{q} D = \Gamma(\Ker(d_2)) / \Gamma(\Z \pi/N).
\]
\begin{lemma}
	\label{lemma:a_1_a_2}
	There exists $\alpha_1, \alpha_2 \in \Gamma(\Ker(d_2))$
	such that
	\begin{enumerate}[\normalfont (i)]
		\item the corresponding elements in the Tate group
		$\widehat{H}_0(\pi;D)$ map to
		$\widehat{j_*}(\alpha_1) = \alpha_{xy} - \alpha_{y}$
		and $\widehat{j_*}(\alpha_2) = \alpha_y$
		in $\widehat{H}_0(\pi;\Gamma((I,2)))$
		\item $N\cdot \alpha_1 = i_*(2 \cdot (N_x \otimes N_x))$
		and $N\cdot \alpha_2 = i_*(n \cdot (N_x \otimes N_x) + 2 \cdot (\sigma \otimes \sigma))$.		
	\end{enumerate}
\end{lemma}

\begin{proof}
	First let 
	\begin{align*} 
		\widetilde{\alpha_y} &= (1+y,-N_x,2) \odot (0,0,y-1) + 2(0,0,y-1)^{\otimes 2} \\
		\widetilde{\alpha_{xy}} &= (1+y,-N_x,2) \odot (0,xy-1,xy-1) + 2(0,xy-1,xy-1)^{\otimes 2}
	 \end{align*}
	be in $\Gamma(\Ker(d_2))$ so that the
	corresponding elements in $\widehat{H}_0(\pi;D)$ have
	\begin{align*}
		\widehat{j_*}(\widetilde{\alpha_y}) &= 2 \odot (y-1) + 2 (y-1) \otimes (y-1) = \alpha_y \\
		\widehat{j_*}(\widetilde{\alpha_{xy}}) &= 2 \odot (xy - 1) + 2 (xy-1) \otimes (xy-1) = \alpha_{xy}.
	\end{align*}
	Let us further in $\Gamma(\Ker(d_2))$ define
	\begin{align*}
		\alpha_1 & = \widetilde{\alpha_{xy}} - \widetilde{\alpha_y} - \beta_1 && \text{ where } 
		\beta_1 = (1-x,xy-1,0)\odot(0,xy-1,xy-1), \\
		\alpha_2 & = \widetilde{\alpha_y} - \beta_2 && \text{ where }
		\beta_2 = (x-1,1-xy,0)\odot(\sigma,-\frac{n}{2}N_x,N_x).
	\end{align*}
	Observe that in $D$, we have $\beta_1 = f(-1 \otimes (xy-1))$.
	It is easy to see that
	$\widehat{j_*}(\alpha_1) = \alpha_{y} - \alpha_{xy}$
	since $j_* \circ f = 0$
	and $\widehat{j_*}(\alpha_2) = \alpha_y$ as $\widehat{j_*}(\beta_2) = 0$,
	which confirms part (i) of the Lemma.

	For part (ii) we make the following computation in $\Gamma(\Ker(d_2))$:
	\begin{align*}
	N \cdot \beta_1 &=N_x(1+xy)((1-x,xy-1,0)\odot(0,xy-1,xy-1))\\
		&=N_x(((1-xy)(1+y),2(xy-1),0)\odot(0,xy-1,xy-1))\\
	N \cdot \widetilde{\alpha_{xy}} &=N_x(1+xy)((1+y,-N_x,2) \odot (0,-x^{-1}(xy-1),xy-1) + 2(0,-x^{-1}(xy-1),xy-1)^{\otimes 2})\\
		&=N_x((1-xy)(1+y),-N_x+yN_x,2(1-xy))\odot(0,xy-1,xy-1)+2N_x(0,xy-1,xy-1)^{\otimes 2}\\
		&=N\beta_1+N_x(0,-N_x+yN_x,0)\odot(0,xy-1,xy-1)\\
		&=N\beta_1+(0,(y-1)N_x,0)\odot(0,0,(y-1)N_x)+2 (0,(y-1)N_x,0)^{\otimes 2}\\
	N\cdot \beta_2&=(1+y)N_x((x-1,1-xy,0)\odot(0,-\frac{n}{2}N_x,N_x))+(1+x)\sigma((x-1,1-xy,0)\odot(\sigma,0,0))\\
		&=(1+y)((0,N_x-yN_x,0)\odot(0,-\frac{n}{2}N_x,N_x))+(1+x)((\sigma(x-1),0,0)\odot(\sigma,0,0))\\
		&=(0,(y-1)N_x,0)\odot(0,0,(y-1)N_x)-n(0,(y-1)N_x,0)^{\otimes 2} - 2(\sigma(x-1,1-yx,0))^{\otimes 2}\\
	N \cdot \widetilde{\alpha_y} &=N_x(1+y)((1+y,-N_x,2) \odot (0,0,y-1) + 2(0,0,y-1)^{\otimes 2})\\
		&=N_x((0,-(1-y)N_x,2(1-y))\odot(0,0,y-1)+4(0,0,y-1)^{\otimes 2})\\
		&=N_x((0,-(1-y)N_x,0)\odot(0,0,y-1))\\
		&=(0,N_x(y-1),0)\odot(0,0,N_x(y-1)).
	\end{align*}
	Hence we have that:
	\begin{align*}
	N \cdot \alpha_1 & = 2 (0,N_x (1-y),0)^{\otimes 2} = i_*(2 \cdot (N_x \otimes N_x)) \\
	N \cdot \alpha_2 & =  n (0,N_x (1-y),0)^{\otimes 2} + 2 (\sigma (x-1,1-xy,0))^{\otimes 2} = i_*( n \cdot (N_x \otimes N_x) + 2 \cdot (\sigma \otimes \sigma))
	\end{align*}
	since $i_*(N_x \otimes N_x) = (0,N_x (1-y),0)^{\otimes 2}$
	and $i_*(\sigma \otimes \sigma) = (\sigma(x-1,1-xy,0))^{\otimes 2}$.
\end{proof}

\begin{proof}[Proof of \Cref{prop:H_0(D)}]
	\label{proof:H_0(D)}
	Let $\alpha_1$, $\alpha_2$ be as in \Cref{lemma:a_1_a_2}.
	We will view them as elements of $\Z \otimes_{\Z \pi} D$
	using the identification $D = \Gamma(\Ker(d_2)) / \Gamma(\Z\pi/N) $.
	In \Cref{lemma:a_1_a_2} (ii), we showed that $N\cdot \alpha_1$, $N\cdot \alpha_2 \in \IM(i_*)$ which implies that $ N\cdot \alpha_1 = N \cdot \alpha_2 = 0 \in D$ and so $\alpha_1, \alpha_2 \in \widehat{H}_0(\pi;D)$.
	
	By \Cref{lemma:f_*=0}, the map
	\[
		\widehat{j_*} \colon \widehat{H}_0(\pi;D) \to \widehat{H}_0(\pi;\Gamma((I,2)))
	\]
	is injective. By \cref{lemma:H_0(G(I2))},
	$\widehat{H}_0(\pi;\Gamma((I,2)))$ is generated by $\alpha_{xy}$ and $\alpha_y$. 
	By \Cref{lemma:a_1_a_2} (i), we have that
	$\widehat{j_*}(\alpha_1) = \alpha_{xy} - \alpha_{y}$ and
	$\widehat{j_*}(\alpha_2) = \alpha_y$.
	Hence $\widehat{j_*}$ is bijective,
	$\widehat{H}_0(\pi;D) \cong \Z/2 \oplus \Z/2$ and is generated by $\alpha_1$, $\alpha_2$.	
	
	Now to compute the boundary map
	$\wh H_0(\pi;D) \xrightarrow{\partial} \wh H_{-1}(\pi;\Gamma(\Z\pi/N))$,
	consider the map of short exact sequences
	\[
	\begin{tikzcd}[column sep = 1cm]
		0 \ar[r] & \Gamma(\Z \pi/N) \ar[d,"N"] \ar[r, "i_*"] & 
		\Gamma(\Ker(d_2)) \ar[d,"N"] \ar[r,"q"] & 
	    D \ar[d,"N"] \ar[r] & 0 \\
		0 \ar[r] & \Gamma(\Z \pi/N) \ar[r, "i_*"] & 
		\Gamma(\Ker(d_2)) \ar[r,"q"] & 
		D \ar[r] & 0
	\end{tikzcd}
	\]	
	Since $N\cdot \alpha_1 = i_*(2 \cdot (N_x \otimes N_x))$
	and $N\cdot \alpha_2 = i_*(n \cdot (N_x \otimes N_x) + 2 (\sigma \otimes \sigma))$, the boundary map $\partial$
	sends $\alpha_1$ and $\alpha_2$ to $2\cdot(N_x\otimes N_x)$ and $n \cdot (N_x \otimes N_x) + 2 (\sigma \otimes \sigma)$, respectively.
\end{proof}

In order to finish the proof of Theorem \ref{thm:ker=0},
we first need the following two lemmas.

\begin{lemma}
	\label{lem:HKLemma2.2}
	Let $ G$ be a finite group and let $\psi \colon \Z  G \to \Z  G/N$ be the quotient map. For each $g \in  G$ of order two, fix a set of coset representatives $\{x_1, \cdots, x_n \}$ for $ G/\langle g \rangle$ and let $\Sigma_{ G/\langle g \rangle} = \sum_{i=1}^n x_i$. Then there is an isomorphism of abelian groups
	\[ 
		\left( \bigoplus_{g \ne 1, g^2 =1} \Z/2 \right) /(1,\cdots,1) \cong \IM(\psi_* \colon \widehat{H}_{-1}( G; \Gamma(\Z G)) \to \widehat{H}_{-1}( G; \Gamma(\Z G/N))	
	\]
	which, on the summand indexed by $g$, has the form
	$1 \mapsto \psi_*(\Sigma_{ G/\langle g \rangle} \cdot (1 \odot g))$. 
\end{lemma}

\begin{proof}
	Let $S$ be the set given by a representative of $g,g^{-1}$ for each $g\in  G, g^2\neq 1$. By \cite[Lemma~2.2]{HK}, we have 
	\[
		\Z G \oplus 
		\bigoplus_{S}\Z G \oplus 
		\bigoplus_{g\ne 1, g^2=1} \Z G / (1-g)\Z G
		\cong \Gamma(\Z G).
	\] 
	On the first summand the isomorphism is given by $h\mapsto h\otimes h$ and on a summand corresponding to $g\in G,g\neq 1$ the isomorphism sends $h$ to $hg\otimes h+h\otimes hg$. On $\Z G$, the norm map $\Z\to (\Z G)^ G$ is an isomorphism, on $\Z G/(1-g)\Z G$ with $g^2=1, g\neq 1$, the norm map $\Z\to (\Z G/(1-g)\Z G)^G$ is injective with cokernel $\Z/2$. The cokernel is generated by summing over some set of representatives of $ G/\langle g\rangle $. As $\wh H_{-1}( G;\Z G)=0$ and $\wh H_{-1}( G;\Z G/(1-g)\Z G)\cong \wh H_{-1}(\langle g\rangle;\Z)\cong \Z/2$, this implies that there is an isomorphism
	\[ 
		\bigoplus_{g \ne 1, g^2 =1} \Z/2 \, \, 
		\cong \bigoplus_{g \ne 1, g^2 =1} \wh H_{-1}( G;\Z G/(1-g)\Z G)
		\cong \wh H_{-1}( G;\Gamma(\Z G)).
	\]
	It can be shown (see \cite[p.~529]{HK-cor}) that, on the summand indexed by $g$, this map is given by:
	\[
		1 \mapsto \Sigma_{ G/\langle g\rangle} \mapsto \Sigma_{ G/\langle g \rangle} \cdot (1 \odot g).
	\] 
	
	Now note that there is an exact sequence $0 \to \Z \to \Z G \to \Z G/N \to 0$
	which has associated sequences from \Cref{lem:bauer}:
	\[
		0 \to \underbrace{\Gamma(\Z)}_{\cong \Z} \to \Gamma(\Z G) \to D_0 \to 0
	\]
	\[ 
		0 \to \underbrace{\Z \otimes_{\Z} (\Z G/N)}_{\cong \Z  G/N} \to D_0 \to \Gamma(\Z G/N) \to 0
	\]	
	We have that
	$\wh H_{-1}(G;\Z G/N)\cong \wh H_{-2}(G;\Z) \cong H^1( G;\Z) =0$
	and so the two long exact sequences for Tate homology can be combined
	at the $\widehat{H}_{-1}( G;D_0)$ term to give an exact sequence:
	\[
		\underbrace{\widehat{H}_{-1}( G;\Z)}_{\cong \Z/| G |} \xrightarrow[]{1 \mapsto N \otimes N}
		\widehat{H}_{-1}( G;\Gamma(\Z G)) \xrightarrow[]{\psi_*}
		\widehat{H}_{-1}( G;\Gamma(\Z G/N)).
	\]
	By exactness, we have that
	\[
		\IM(\psi_* \colon 
		\widehat{H}_{-1}( G; \Gamma(\Z G)) \to \widehat{H}_{-1}( G; \Gamma(\Z G/N))) \cong
		\widehat{H}_{-1}( G;\Gamma(\Z G)) / N \otimes N.
	\]
	Let $\{x_1(g), \cdots, x_n(g)\}$ be coset representatives for $ G/\langle g \rangle$ where $g \ne 1$, $g^2=1$. It can be shown (again, see \cite[p.~529]{HK-cor}) that: 
	\[
		N \otimes N 
		= 
		N \cdot \gamma
		+
		\sum_{g\ne 1, g^2 =1} \Sigma_{ G/\langle g \rangle} \cdot (1 \odot g)
	\]
	for some $\gamma \in \Gamma(\Z G/N)$, and so $N \otimes N$ maps to the diagonal element under the isomorphism described above.
\end{proof}

In the case where $\pi$ is the dihedral group of order $2n$ for $n$ even, the non-trivial order two elements are
\[ 
	\{ yx^i : 0 \le i < n\} \cup \{x^{n/2}\}
\]
and we can take $\Sigma_{\pi/\langle yx^i \rangle} = N_x$ for all $0 \le i < n$
and $\Sigma_{\pi / \langle x^{n/2} \rangle} = (1+y)\sum_{i=0}^{n/2-1}x^i$.

\begin{lemma} \label{lemma:LI}
	The elements $2 \cdot (N_x\otimes N_x)$, $2 \cdot (\sigma \otimes \sigma) \in \wh H_{-1}(\pi;\Gamma(\Z\pi/N))$ are linearly independent.
\end{lemma}
\begin{proof}
	Consider the exact sequence
	$\Z\to \Z\pi\to \Z\pi/N$ and the associated exact sequences
	$\Gamma(\Z)\to \Gamma(\Z\pi)\to D_0$ 
	and $\Z\otimes_\Z\Z\pi/N \to D_0 \to \Gamma(\Z\pi/N)$.
	A pre-image of $2(N_x\otimes N_x) \in D_0$ in $\Gamma(\Z\pi)$
	is given by
	\[
		2(N_x\otimes N_x) - N \odot N_x
		=
		- yN_x \odot N_x.
	\]
	Note that on $yN_x\otimes N_x\in \Z\pi \otimes_{\Z} \Z\pi$ the element $x$ acts trivial and that $yN_x\otimes N_x$
	is mapped to $N_x\otimes yN_x$ under $y$.
	Hence $- yN_x \odot N_x$ is a fixed point in $\Gamma(\Z\pi)$
	and thus represents an element of $\wh H_{-1}(\pi;\Gamma(\Z\pi))$. 
	
	Under the isomorphism from \cref{lem:HKLemma2.2}, the element
	\[
		y N_x \odot N_x = N_x \odot y N_x = \sum_{i,j} x^j \odot y x^{i+j} = \sum_i N_x \cdot (1 \odot yx^i)
	\]
	maps to $1$
	in all summands index by $yx^i$ for some $i$ and to 0 in all other summands.
	In particular, $2(N_x\otimes N_x)$ is non-trivial in $\wh H_{-1}(\pi;\Gamma(\Z\pi/N))$.
	
	We have with our usual notation $\sigma=(1+yx)\sum_{i=1}^{n/2}x^{2i}$ that
	\[
		2(\sigma \otimes \sigma)-N \otimes N + N \odot x\sigma
		= \sigma \otimes \sigma+x\sigma \otimes x\sigma
	\]
	which is a fixed point in $\Gamma(\Z\pi)$.
	We claim that, under the isomorphism from \cref{lem:HKLemma2.2},
	this element maps to $1$ in all summands index by $yx^{2i+1}$ for some $i$
	and to 0 in all summands index by $yx^{2i}$ for some $i$.
	It also maps to $1$ in the summand indexed by $x^{n/2}$ if and only if $n/2$ is even.
	Let $N_{x^2}:=\sum_{i=1}^{n/2}x^{2i}$, so that $\sigma = (1+yx) N_{x^2}$. Then
	\begin{align*}
		\sigma\otimes\sigma+x\sigma\otimes x\sigma&=(1+x)(\sigma\otimes\sigma)\\
		&=(1+x)(N_{x^2}\otimes N_{x^2}+yxN_{x^2}\otimes yxN_{x^2}+N_{x^2}\odot yxN_{x^2}).
	\end{align*}
	We have \[(1+x)(N_{x^2}\odot yxN_{x^2})=N_x(1\odot yxN_{x^2})=\sum_{i=1}^{n/2}N_x(1\odot yx^{2i+1}).\]
	Furthermore,
	\begin{align*}
		(1+x)(N_{x^2}\otimes N_{x^2}+yxN_{x^2}\otimes yxN_{x^2})&=(1+x)(1+yx)(N_{x^2}\otimes N_{x^2})\\
		&=N(1\otimes N_{x^2})=\sum_{i=1}^{n/2}N(1\otimes x^{2i}).
	\end{align*}
	Note that $N(1\otimes x^{2i})=Nx^{n-2i}(1\otimes x^{2i})=N(x^{n-2i}\otimes 1)$ and thus $N(1\otimes x^{2i}+1\otimes x^{n-2i})=N(1\odot x^{2i})$. Hence if $n/2$ is odd, we have
	\[\sum_{i=1}^{n/2}N(1\otimes x^{2i})=\sum_{i=0}^{(n+2)/4}N(1\odot x^{2i})\]
	which is trivial in $\wh H_{-1}(\pi;\Gamma(\Z\pi/N))$.
	If $n/2$ is even, we have
	\[\sum_{i=1}^{n/2}N(1\otimes x^{2i})=N(1\otimes x^{n/2})+\sum_{i=0}^{n/4-1}N(1\odot x^{2i}).\]
	The last summand is again trivial in $\wh H_{-1}(\pi;\Gamma(\Z\pi/N))$ and we have
	\[
		N(1\otimes x^{n/2})
		=(1+y)\sum_{i=0}^{n/2-1}x^i(1\odot x^{n/2})
	\]
	which maps to the summand index by $x^{n/2}$ under the isomorphism from \cref{lem:HKLemma2.2}.
	This proves the claim.
	As $2(N_x\otimes N_x)$ maps to $1$ in all summands index by $yx^i$ for all $i$,
	the two elements are linearly independent.
\end{proof}

\begin{proof}[Proof of Theorem \ref{thm:ker=0}]
	\label{proof:ker=0}
	We showed previously that there is an exact sequence:
	\[ 
	0 \to 
	\widehat{H}_0(\pi;\Gamma(\Ker(d_2))) \xrightarrow[]{\widehat{q}}
	\widehat{H}_0(\pi;D) \xrightarrow[]{\partial}
	\widehat{H}_{-1}(\pi;\Gamma(\Z \pi/N)) \to \ldots.
	\]
	By \Cref{prop:H_0(D)}, we have that $\widehat{H}_{0}(\pi;D)$ is generated by $\alpha_1$ and $\alpha_2$ with
	$\partial(\alpha_1) = 2 \cdot (N_x \otimes N_x)$ and $\partial(\alpha_2) = n \cdot (N_x \otimes N_x) + 2 \cdot (\sigma \otimes \sigma)$.
	As $n$ is even, $n \cdot (N_x\otimes N_x)$ is a multiple of $2 \cdot (N_x\otimes N_x)$.
	By \Cref{lemma:LI}, $2 \cdot (N_x \otimes N_x)$ and $2 \cdot (\sigma \otimes \sigma)$ are linearly independent and so $\partial$ is injective.
	By exactness, this implies that $\widehat{H}_0(\pi;\Gamma(\Ker(d_2)))=0$.	
\end{proof}

\section{An explicit parametrisation for \texorpdfstring{$\Omega^3(\Z)$}{the cokernel}}
\label{sec:coker1}

Since $\coker(d^2) \cong \Ker(d_2)^*$, from dualizing
\Cref{prop:kernel_exact_sequence}
there is an exact sequence of
left $\Z \pi$-modules:
\[ 
	0 \to (I,2)^*
	\xrightarrow[]{j^*} \coker(d^2)
	\xrightarrow[]{i^*} (\Z \pi /N)^*
	\to 0
\]
where we recall that the original maps from the kernel
sequence were
$i = \cdot \left(\begin{smallmatrix} x-1 & 1-xy & 0 \end{smallmatrix}\right)$ and 
$j = \cdot \left(\begin{smallmatrix} 0 \\ 0 \\ 1 \end{smallmatrix}\right)$.
Dualizing preserves exactness of the sequence
since all modules are $\Z \pi$-lattices,
as discussed for example in \cite[Remark 1.8]{Ni20-I}.
Our aim will now be to simplify each of the terms in the sequence above.

We first note that $d^2$ as the dual of $d_2$ is given by transposing
the matrix for $d_2$ and applying the involution. That is, 
$d^2=\cdot 
\left(
\begin{smallmatrix}
N_x & 1+xy & 0 \\ 
-(1+y) & x^{-1}-1 & 1+y
\end{smallmatrix}
\right)$. 
With the same procedure, the dual of
$\Z \pi / N 
\xrightarrow[i]{\cdot \left(\begin{smallmatrix} x-1 & 1-xy & 0 \end{smallmatrix}\right)}
\Ker(d_2)$
is given by
$i^{*} = \cdot \left(\begin{smallmatrix}
x^{-1}-1 \\ 1-xy \\ 0
\end{smallmatrix}\right)$. 

To reduce the number of inverses in the following computation,
we substitute $x^{-1}$ by $x$ and obtain
$d^2=\cdot 
\left(
\begin{smallmatrix}
N_x & 1+yx & 0 \\ 
-(1+y) & x-1 & 1+y
\end{smallmatrix}
\right)$,
and the map $\coker(d^2) \to (\Z\pi/N)^*$ is given by
$i^{*} = \cdot \left(\begin{smallmatrix}
x-1 \\ 1-yx \\ 0
\end{smallmatrix}\right)$. This gives the following exact sequence.
\begin{equation}
\label{eq:firstseqcoker}
0\to (I,2)^*\xrightarrow{\cdot \left(\begin{smallmatrix}
	0 & 0 & 1
	\end{smallmatrix}\right)}
\coker(d^2)\xrightarrow{\cdot \left(\begin{smallmatrix}
	x-1 \\ 1-yx \\ 0
	\end{smallmatrix}\right)}
(\Z\pi/N)^*\to0
\end{equation}

\begin{lemma}
	\label{lem:dual_I_2}
	There is an isomorphism of $\Z \pi$-modules
	\[
		\varphi \colon (N,2) \to (I,2)^*
	\]
	which sends
	$2 \mapsto i_{(I,2);\Z \pi}$,
	$N \mapsto p$,
	where $i_{(I,2);\Z \pi} \colon (I,2) \hookrightarrow \Z \pi$ is inclusion
	and $p \colon (I,2) \rightarrow \Z \pi$ is given by
	$p(\lambda) = N \cdot \frac{1}{2} \varepsilon(\lambda)$.
\end{lemma}

\begin{remark}
	The definition of the map $p$ makes sense since $\varepsilon((N,2)) \subseteq 2 \Z$.
	By abuse of notation, we could also write $p = \frac{1}{2} N \varepsilon$.
\end{remark}

\begin{proof}
	First recall that $\Z \pi /N \cong I^*$ which sends $1$ to the inclusion map
	$i_{I; \Z \pi} \colon I \hookrightarrow \Z \pi$, 
	and so $I^*$ as a $\Z \pi$-module is generated by $i_{I; \Z \pi}$.
	By dualising the exact sequence
	$0 \to I \hookrightarrow (I,2) \xrightarrow[]{\varepsilon} 2 \Z \to 0$,
	we get
	\[
		0 \to
		\Z \xrightarrow[]{1 \mapsto p}
		(I,2)^* \xrightarrow[]{(\lambda i_{I; \Z \pi} \mapsto \lambda) \circ i_{I;(I,2)}} 
		\Z \pi/N \to 0
	\]
	where $i_{I;(I,2)} \colon I \hookrightarrow (I,2)$ is the inclusion map.
	Since under the second map
	$(i_{(I,2);\Z \pi} \colon (I,2) \hookrightarrow \Z \pi) \mapsto 1 \in \Z \pi/N$
	is a generator, this implies that $(I,2)^* = \langle i_{(I,2);\Z \pi}, p \rangle$.
	To see that $\varphi$ is well-defined, note that
	$N \cdot i_{(I,2);\Z \pi} = 2 \cdot p$.
	Hence $\varphi$ is a surjective $\Z \pi$-module homomorphism
	and so it remains to show injectivity.
	
To see this, note that the underlying abelian groups of $(N,2)$ and $(I,2)$ are both torsion-free and have rank $|\pi|$ since $\Q \otimes_{\Z} (N,2) = \Q \otimes_{\Z} (I,2) = \Q \pi$. This implies that the underlying abelian group of $(I,2)^*$ is also torsion-free of rank $|\pi|$. Hence $\varphi$ is bijective since every surjection $\varphi: \Z^{|\pi|} \to \Z^{|\pi|}$ is also a bijection.
\end{proof}

\begin{lemma}
	\label{lem:dual_Zpi_mod_N}
	There is an isomorphism of $\Z \pi$-moduless
	$(\Z \pi /N)^* \xrightarrow{\cong} I$
	which sends the map $f \in (\Z \pi /N)^*$ to $f(1) \in I$.
\end{lemma}

\begin{proof}
	To show this we can, for example, dualise the isomorphism $\Z \pi /N \cong I^*$
	which sends $1$ to the inclusion map $I \hookrightarrow \Z \pi$.
\end{proof}

We can now substitute $(I,2)^*\cong (N,2)$ and $(\Z\pi/N)^*\cong I$ in \eqref{eq:firstseqcoker}.
For this we need to find an element in $\coker(d^2)$
which becomes $(0,0,N)$ under multiplication by $2$. 
We compute
\[
	2(N_x,0,0)-(0,0,N)
	=(1-yx)(N_x,1+yx,0) - N_x(-(1+y),x-1,1+y)
\]
in $\Z\pi^3$. Since $(N_x,1+yx,0)$ and $(-(1+y),x-1,1+y)$ are in the image of $d^2$, this implies $(0,0,N)=2(N_x,0,0)\in \coker(d^2)$.
Thus the following proposition follows from applying
\cref{lem:dual_I_2,lem:dual_Zpi_mod_N} to \eqref{eq:firstseqcoker}.
\begin{prop}
	\label{prop:cokerrep}
	With respect to the above identification of $\coker(d^2)$, there is an exact sequence:
	\[ 
		0 \to 
		(N,2) \xrightarrow[i']{\substack{2 \mapsto (0,0,1) \\ N \mapsto (N_x,0,0)}} 
		\coker(d^2) \xrightarrow[j']{ - \cdot \left(\begin{smallmatrix} x-1 \\ 1-yx \\ 0 \end{smallmatrix}\right)} 
		I \to 0.
	\]
	Furthermore, we have $j'(1,0,0) = x-1$ and $j'(-y,-1,0) = y-1$ which gives lifts of the $\Z \pi$-module generators for $I$.
\end{prop}

\section{Computing $\widehat{H}_0(\pi;\Gamma(\coker(d^2)))$}
\label{sec:coker2}

At the end of this section on page~\pageref{proof:coker=0}
we will prove the following:
\begin{thm}
	\label{thm:coker=0}
	If $\pi$ is a dihedral group of order $2n$ for $n$ even,
	then $\widehat{H}_0(\pi;\Gamma(\coker(d_2)))=0$.	
\end{thm}

Let $E = \Gamma(\coker(d^2)) / \Gamma((N,2)) $
so that there is an exact sequence
\[
	0 \to \Gamma((N,2)) \xrightarrow[]{i'_*} \Gamma(\coker(d^2)) \xrightarrow[]{q'} E \to 0
\]
where $q'$ is the quotient map. By \cref{lem:bauer}, there is an exact sequence:
\[ 0 \to (N,2) \otimes_{\Z} I \xrightarrow[]{f'} E \xrightarrow[]{j'_*} \Gamma(I) \to 0.\]
By \cref{prop:cokerrep},
the map $f'$ is defined by
\begin{align*}
	f' \colon (N,2) \otimes_{\Z} I & \to E = \Gamma(\coker(d^2)) / \Gamma((N,2)) \\
	2 \otimes (x-1) &\mapsto [(0,0,1) \odot (1,0,0)]\\
	N \otimes (x-1) &\mapsto [(N_x,0,0) \odot (1,0,0)]\\
	2 \otimes (y-1) &\mapsto [(0,0,1) \odot (-y,-1,0)]\\
	N \otimes (y-1) &\mapsto [(N_x,0,0) \odot (-y,-1,0)]
\end{align*}
By the long exact sequence for Tate homology applied to the first exact sequence, we have:
\begin{equation}
	\label{eqn:les_cokernel}
	\ldots \to
	\widehat{H}_0(\pi; \Gamma((N,2))) \xrightarrow[]{\widehat{i'_*}}
	\widehat{H}_0(\pi; \Gamma(\coker(d^2))) \xrightarrow[]{\widehat{q'}} 
	\widehat{H}_0(\pi;E) \to \ldots.
\end{equation}
We will now aim to show the following.
\begin{prop} 
	\label{prop:H_0(E)}
	$\widehat{H}_0(\pi;E)=0$.
\end{prop}

We begin by noting that $\widehat{H}_0(\pi ; \Gamma(I)) = 0$ by \cite[Theorem~2.1]{HK}.
By the long exact sequence on Tate homology for the second exact sequence, we thus have an exact sequence:
\[ 
	\ldots \to
	\widehat{H}_1(\pi;\Gamma(I)) \xrightarrow[]{\partial}
	\widehat{H}_0(\pi; (N,2) \otimes_{\Z} I) \xrightarrow[]{\widehat{f'}}
	\widehat{H}_0(\pi;E) \to
	0 \to \ldots
\]
where $\partial$ denotes the boundary map. 
Hence, in order to show Proposition \ref{prop:H_0(E)},
it will suffice to prove that $\partial$ is surjective.

\begin{lemma}
	\label{lem:cokernel_first}
	For every finite group $G$ there is an isomorphism of abelian groups
	\[ 
		G^{\textup{ab}} \otimes_{\Z} \Z/2 
		\to \widehat{H}_0(G ; (N,2) \otimes_{\Z} I)
	\]
	given by $g \mapsto N \otimes (g-1)$.
\end{lemma}

\begin{proof}
	Similarly to the proof of Lemma \ref{lemma:dim-shift},
	we consider the following two exact sequences.
	Firstly the sequence
	$0 \to I \to \Z G \xrightarrow{\varepsilon} \Z \to 0$
	tensored with $(N,2) \otimes_{\Z} -$:
	\[ 
		0 \to
		(N,2) \otimes_{\Z} I \to
		(N,2) \otimes_{\Z} \Z G
		\xrightarrow[]{\id \otimes \varepsilon}
		(N,2) \otimes_{\Z} \Z
		\to 0
	\]
	where the middle term is free by \cite[Lemma 4.3]{KPR20}.
	And secondly, we have:
	\[ 
		0 \to \Z G \xrightarrow[]{\cdot 2} (N,2) \xrightarrow[]{N \mapsto 1} \Z/2 \to 0.
	\]
This is exact since $N \cdot \Z G = N \cdot \Z$ and, by the second isomorphism theorem for modules, we have that $(N,2)/2\cdot \Z G\cong N\cdot \Z / (2 \cdot \Z G \cap N \cdot \Z) = N \cdot \Z / 2N \cdot \Z \cong \Z/2$.

	By applying dimension shifting twice, we get:
	\[ 
		\widehat{H}_0(G ; (N,2) \otimes_{\Z} I)
		\cong \widehat{H}_1(G; (N,2))
		\cong \widehat{H}_1(G; \Z/2)
	\]
	and $\widehat{H}_1(G;\Z/2) \cong G^{\text{ab}}\otimes_{\Z} \Z/2$.
	
	Under the isomorphism
	$G^{\text{ab}}\otimes_{\Z} \Z/2\cong \wh H_1(G;\Z/2)\cong \wh H_1(G;(N,2))$, and adopting \Cref{convention}, the element $g$ maps to $[c_{g} \otimes N]$
	where $c_{g} \in C_1$ is such that $d_1(c_{g})=g-1\in C_0$.
	Under the boundary map $\wh H_1(G;(N,2))\to \wh H_0(G;(N,2)\otimes_\Z I)$ induced by the first exact sequence above, the element maps to $N \otimes (g-1)$ as claimed.
\end{proof}

We can now show the following which completes the proof of \Cref{prop:H_0(E)}. 
Recall that, for the presentation
$\mathcal{P} = \langle x, y \mid x^n y^{-2}, xyxy^{-1}, y^2 \rangle$,
we obtained a partial free resolution $C_*(\mathcal{P})$ using Fox derivatives.
In what follows, we will write
$(C_*, d_*) = (C_*(\mathcal{P}), d_*)$ for $0 \le * \le 2$
and will adopt \Cref{convention} using this specific resolution.

\begin{lemma}
	\label{lemma:cokernel_first_half}
	The boundary map
	$\partial \colon \wh H_1(\pi;\Gamma(I)) \to \wh H_0(\pi;(N,2)\otimes_\Z I)$
	is surjective.
\end{lemma}

\begin{proof}
	For each $g \in \pi$ of order 2, let $c_g \in C_1$ be such that $d_1(c_g)=1-g$.
	Note that the map
	\[ 
		d_1 \otimes \id_{\Gamma(I)} \colon
		\Z \pi^2 \otimes_{\Z \pi} \Gamma(I) \to \Z \pi \otimes_{\Z \pi} \Gamma(I) \cong \Gamma(I)
	\]
	sends $c_g \otimes ((1-g)^{\otimes 2}) \mapsto (1-g) \cdot (1-g)^{\otimes 2} = 0$
	and so we have defined an element
	\[
		\gamma_g = [c_g\otimes (1-g)^{\otimes 2}] \in \wh H_1(\pi;\Gamma(I)).
	\]
	Now note that $\partial(\gamma_g) \in \wh H_0(\pi;(N,2)\otimes_\Z I)$
	is defined by a diagram chase on the following diagram:
	\[
	\begin{tikzcd}[column sep = 1cm]
	0 \ar[r] & C_1 \otimes_{\Z \pi} ((N,2) \otimes_{\Z} I) \ar[d,"d_1 \otimes \id"] \ar[r,"\id \otimes f'"] &
	C_1 \otimes_{\Z \pi} E \ar[d,"d_1 \otimes \id"] \ar[r,"\id \otimes j'_*"] &
	C_1 \otimes_{\Z \pi} \Gamma(I)\ar[d,"d_1 \otimes \id"] \ar[r] & 0\\
	0 \ar[r] & (N,2) \otimes_{\Z} I \ar[r,"f'"] & E \ar[r,"j'_*"] & \Gamma(I) \ar[r] & 0
	\end{tikzcd}
	\]
	where we use the identification $C_0 \otimes_{\Z \pi} M \cong M$ in the bottom exact sequence. 
	
	It will be useful to note that, if $w_g \in \coker(d^2)$ is a lift of $1-g \in I$, then
	\begin{align*}
		& (\id \otimes j'_*)(c_g \otimes [w_g \otimes w_g]) = \gamma_g \text{ and} \\
		& (d_1 \otimes \id)(c_g \otimes [w_g \otimes w_g]) = (1-g) \cdot [w_g \otimes w_g].
	\end{align*}
	
	We will now show that $\partial(\gamma_{yx}) = N \otimes (yx -1)$ and $\partial(\gamma_y) = N \otimes (y-1)$. This finishes the proof since, by Lemma \ref{lem:cokernel_first} and the fact that $\pi$ is generated by $yx$ and  $y$, the elements $N \otimes (yx -1)$ and $N \otimes (y -1)$ are generators for $\wh H_0(\pi; (N,2) \otimes_{\Z} I)$.
	
	We will begin by computing $\partial(\gamma_{yx})$.
	Since $w_{yx} = (0,1,0) \in \coker(d^2)$ is a lift of $1-yx \in I$,
	it suffices to prove that $f'(N \otimes (yx-1)) = (1-yx) \cdot [(0,1,0) \otimes (0,1,0)]$.
	Firstly, since $(0,yx,0) \in \coker(d^2)$ maps to $yx-1 \in I$, we can take
	$f'(N \otimes (yx-1)) = [(N_x,0,0)\odot(0,yx,0)]$.
	Secondly, note that $(N_x,1+yx,0) = 0 \in \coker(d^2)$ and so we have:
	\[ 
		(0,1,0) = -(N_x,0,0) -(0,yx,0) \in \coker(d^2).
	\]
	By using this repeatedly inside $E$, we get:
	\begin{align*}
		 (1-yx) \cdot [(0,1,0)\otimes (0,1,0)] = & [(0,1,0)\otimes (0,1,0)]-[(0,yx,0)\otimes (0,yx,0)] \\
		= & -[(0,1,0)\otimes (N_x,0,0)]+[(N_x,0,0)\otimes(0,yx,0)] \\
		= & [(N_x,0,0)^{\otimes 2}]+[(N_x,0,0)\odot(0,yx,0)] \\
		= & [(N_x,0,0)\odot(0,yx,0)]
	\end{align*}
	where we have used for the last equality
	that $[(N_x,0,0)^{\otimes 2}]=0 \in E$ since $i'(N) = (N_x,0,0)$.

	We will now compute $\partial(\gamma_{y})$.
	Similarly, we can take $w_y = (y,1,0)$ to be a lift of $y-1 \in I$
	so that $f'(N \otimes (y-1)) = [(N_x,0,0) \odot (y,1,0)]$. We now compute:
		\begin{align*}
		&(1-y) \cdot [(y,1,0)\otimes(y,1,0)]\\
		&=[(y,1,0)\otimes(y,1,0)]-[(1,y,0)\otimes(1,y,0)]\\
		&=[(y,1,0)\otimes(y,1,0)]+[(1,y,0)\otimes(y,1,0)]-[(1,y,0)\otimes(1+y,1+y,0)]\\
		&=[(1+y,1+y,0)\otimes(y,1,0)]-[(1,y,0)\otimes(1+y,1+y,0)]\\
		&=[(1+y,1+y,0)\odot (y,1,0)]-[(1+y,1+y,0)^{\otimes 2}] \\
		&=[(1+y,1+y,0)\odot (y,1,0)]
		\end{align*}
	where we have used in the last step
	that $[(1+y,1+y,0)^{\otimes 2}] = 0 \in E$ since $j'(1+y,1+y,0)=0$ and so $(1+y,1+y,0) \in \IM(i')$.
	Now note that
	\[
		(y+1,y+1,0)=(0,y+x,1+y)=(-yN_x,0,1+y)=(-N_x,0,1+y)\in \coker(d^2),
	\]
	where the first uses that $(-1-y,x-1,1+y)$ is trivial, the second uses that $y+x=y(1+yx)$ and that $(N_x,1+yx,0)$ is trivial, and the last line uses that $0=(1-yx)(N_x,1+yx,0)=(N_x-yN_x,0,0)$. In particular, this shows that:
	\[
		(1-y) \cdot [(y,1,0)\otimes(y,1,0)] = f'(N \otimes (y-1)) + [(0,0,1+y) \odot (y,1,0)].
	\]	
	Now note that $(\id \otimes j'_*)(c_y \otimes [(0,0,1) \odot (1,y,0)]) =0$ and 
	\begin{align*}
		(d_1 \otimes \id)(c_y \otimes [(0,0,1) \odot (1,y,0)])
		& = (1-y) \cdot [(0,0,1) \odot (1,y,0)] \\
		& =-[(0,0,1+y) \odot (y,1,0)] +[(0,0,1) \odot (1+y,1+y,0)] \\
		& = -[(0,0,1+y) \odot (y,1,0)]
	\end{align*}
	since $j'(0,0,1)=0$ and $j'(1+y,1+y,0)=0$ implies that $(0,0,1), (1+y,1+y,0) \in \IM(i')$
	and so $[(0,0,1) \odot (1+y,1+y,0)] = 0 \in E$.
	
	Hence, if we take
	$\wt \gamma_y = c_y \otimes [(0,0,1) \odot (1,y,0)] + c_y \otimes [(y,1,0) \otimes (y,1,0)] \in C_1 \otimes E$
	to be our lift of $\gamma_y \in C_1 \otimes \Gamma(I)$,
	then $(d_1 \otimes \id)(\wt \gamma_y) = f'(N \otimes (y-1))$ and so $\partial(\gamma_y) = N \otimes (y-1)$, as required.
\end{proof}

In order to prove Theorem \ref{thm:coker=0},
we will now calculate $\widehat{H}_0(\pi; \Gamma((N,2)))$.
Recall that there is an exact sequence:
\[
	0 \to \Z \xrightarrow[]{N} (N,2) \xrightarrow[]{2 \mapsto 1} \Z \pi/N \to 0.
\]
Let $F = \Gamma((N,2)) / \Gamma(\Z) $ so that 
\begin{equation}
	\label{eqn:seq_F}
	0 \to \underbrace{\Gamma(\Z)}_{\cong \Z} \xrightarrow[]{N_*} \Gamma((N,2)) \xrightarrow[]{q_0} F \to 0
\end{equation}
where $q_0$ is the quotient map. By \cref{lem:bauer} again, we get the exact sequence:
\[
	0 \to \underbrace{\Z \otimes_{\Z} (\Z \pi/N)}_{\cong \Z \pi/N} \xrightarrow[]{f_0} F \to \Gamma(\Z \pi/N) \to 0
\]
where $f_0 \colon \Z \pi/N \to F$ sends $1 \mapsto [2 \odot N]$.

\begin{lemma} \label{lemma:H_0(F)}
	$\widehat{H}_0(\pi ; F)$ is generated by $[2 \odot N]$.	
\end{lemma}

\begin{proof}
	First note that $\Z \otimes_{\Z \pi} \Z \pi/N \cong \Z / | \pi | \cong \Z /2n$ and so
	\[
		\widehat{H}_0(\pi ; \Z \pi/N) = \Tors(\Z \otimes_{\Z \pi} \Z \pi/N) \cong \Z/2n.
	\]
	By \cite[Theorem 2.1]{HK}, we have that $\widehat{H}_0(\pi; \Gamma(\Z \pi/N))=0$ and so the map
	\[ 
		\widehat{f_0} \colon \Z/2n \cong \Tors(\Z \otimes_{\Z \pi} \Z \pi /N) \to \Tors(\Z \otimes_{\Z \pi} F)
	\]
	is surjective. Hence $\widehat{f_0}(1) = 2 \odot N$ is a generator
	of $\widehat{H}_0(\pi ; F)$.
\end{proof}

\begin{lemma}
	\label{lemma:H_0G(N,2)}
	$\widehat{H}_0(\pi; \Gamma((N,2))) = \Tors(\Z \otimes_{\Z \pi} \Gamma((N,2)))$ is generated by 
	\[
		\alpha = \frac{n}{2} \cdot (2 \odot N) - N \otimes N.
	\]
\end{lemma}

\begin{proof}
	Since $\Z \otimes_{\Z \pi} \Gamma(\Z) \cong \Z$ is torsion-free,
	the long exact sequence on Tate homology coming
	from the exact sequence \eqref{eqn:seq_F} is:
	\[
		0 \to
		\widehat{H}_0(\pi ; \Gamma((N,2))) \xrightarrow[]{\widehat{q_0}}
		\widehat{H}_0(\pi ; F) \xrightarrow[]{\partial}
		\widehat{H}_{-1}(\pi;\Gamma(\Z)) \to \ldots.
	\]
	By dimension shifting, we have that
	\[ \widehat{H}_{-1}(\pi;\Gamma(\Z)) \cong \widehat{H}_{-1}(\pi;\Z) \cong \widehat{H}_0(\pi ; \Z \pi/N) \cong \Z/2n\]
	and, with respect to the identification
	$\widehat{H}_{-1}(\pi;\Gamma(\Z)) = \Gamma(\Z)^{\pi}/\IM(N)$,
	it is generated by $1 \otimes 1$.
	It follows from a straightforward diagram chase that
	$\partial([2 \odot N]) = 4 \cdot (1 \otimes 1)$.
	Since $[2 \odot N] \in \widehat{H}_0(\pi;F)$ is a generator
	by \Cref{lemma:H_0(F)}, this implies that
	$\ker(\partial) = \langle \frac{n}{2} \cdot [2 \odot N] \rangle$.
	
	Let $\alpha = \frac{n}{2} \cdot (2 \odot N) - N \otimes N \in \Z \otimes_{\Z \pi} \Gamma((N,2))$. Then
	\[ 
		4 \alpha = 2 (N \odot N) - 4(N \otimes N) = 0
	\]
	and so $\alpha \in \Tors(\Z \otimes_{\Z \pi} \Gamma((N,2)))$.
	Since $\widehat{q_0}(\alpha) = \frac{n}{2} \cdot [2 \odot N]$,
	this implies that $\alpha$ generates $\widehat{H}_0(\pi; \Gamma((N,2)))$.
\end{proof}

\begin{remark}
	We were not able to detect whether the generator $\alpha$ is
	non-zero in $\widehat{H}_0(\pi; \Gamma((N,2)))$, and hence we do
	not know whether the module $\widehat{H}_0(\pi; \Gamma((N,2)))$ is trivial.
	Nevertheless, we can finish the proof of \Cref{thm:coker=0} by showing that
	the generator $\alpha$ maps to zero under the map
	$\widehat{i'_*} \colon \widehat{H}_0(\pi; \Gamma((N,2))) \to \widehat{H}_0(\pi; \Gamma(\coker(d^2)))$
	in the long exact sequence \eqref{eqn:les_cokernel}.
\end{remark}

\begin{lemma} 
	\label{lemma:i_+=0}
	$\widehat{i'_*} =0$, i.e.\
	$\widehat{i'_*}(\alpha) = 0 \in \widehat{H}_0(\pi; \Gamma(\coker(d^2)))$.
\end{lemma}

\begin{proof}
	By Lemma \ref{lemma:H_0G(N,2)}, we know that $\widehat{H}_0(\pi; \Gamma((N,2)))$ is generated by $\alpha$
	and so $\IM(\widehat{i'_*})$ is generated by
	\[ 
		\wh \alpha = \widehat{i'_*}(\alpha) 
		= \frac{n}{2} \cdot ((0,0,1) \odot (N_x,0,0)) - (N_x,0,0)^{\otimes 2}.
	\]
	We will now show that $\wh \alpha$ is trivial. Consider the following elements in $\Gamma(\coker(d^2))$:
	\begin{align*}
		c_1 &= (N_x,0,0)^{\otimes 2} + (N_x,0,0)\odot(0,yx,0) \\
		    &= (1-yx) \cdot (0,1,0)^{\otimes 2}
	\end{align*}
	and
	\begin{align*}
	c_2 &= (-N_x,0,0)\odot (y,1,0) -(-N_x,0,1+y)^{\otimes 2} +(0,0,1)\odot(-N_x,0,1+y) \\ 
	&= (1-y) \cdot ((0,0,1) \odot (1,y,0) + (y,1,0)^{\otimes 2})
	\end{align*}
	where the second equalities follow from the calculations in \Cref{lemma:cokernel_first_half}.
	Hence we have that the classes represented by $c_1, c_2 \in \wh H_0(\pi;\Gamma(\coker(d^2)))$
	are trivial.
	
	Since $(N_x,0,0) = (yN_x,0,0) \in \coker(d^2)$, we have in $\coker(d^2)$:
	\begin{align*}
	yx \cdot c_1 + c_2
	&=(-N_x,0,0) \odot (y,0,0) +(-N_x,0,1+y) \odot (0,0,1) +(N_x,0,0)^{\otimes 2} -(-N_x,0,1+y)^{\otimes 2} \\
	&= (-N_x,0,0) \odot (y,0,0) + (N_x,0,0)\odot(0,0,y) + (0,0,1+y) \odot (0,0,1) - (0,0,1+y)^{\otimes 2}\\
	&= (-N_x,0,0)\odot(y,0,0)
		+(N_x,0,0)\odot(0,0,y)-(0,0,y)^{\otimes 2} +(0,0,1)^{\otimes 2}
	\end{align*}
	
	Let $v_1 = (-N_x,0,0)\odot(y,0,0)+(N_x,0,0)\odot(0,0,y)$.
	Since
	\[
		c_3 = -(0,0,y)^{\otimes 2} + (0,0,1)^{\otimes 2} = 0 \in \wh H_0(\pi;\Gamma(\coker(d^2))),
	\]
	the above implies that $v_1 = yx \cdot c_1 + c_2 - c_3 = 0 \in \wh H_0(\pi;\Gamma(\coker(d^2)))$.
	
	Let $S:=\sum_{i=0}^{n/2-1}x^i$ 
	and let $v_2 = (1-x^{n/2}) \cdot (S,0,0)^{\otimes 2}$
	so that $v_2 = 0 \in \wh H_0(\pi;\Gamma(\coker(d^2)))$.
	Now, we have:
	\begin{align*}
	v_2
	&=(S,0,0)^{\otimes 2}-(N_x-S,0,0)^{\otimes 2}\\
	&=(N_x,0,0)\odot (S,0,0)-(N_x,0,0)^{\otimes 2}.
	\end{align*}
	Using $(N_x,0,0) = (yN_x,0,0) \in \coker(d^2)$ again,
	\begin{align*}
		Sy \cdot v_1 + v_2 
		&= (-yN_x,0,0)\odot(S,0,0)+(yN_x,0,0)\odot(0,0,S) + v_2 \\
		&= (-N_x,0,0)\odot(S,0,0)+(N_x,0,0)\odot(0,0,S) + v_2 \\
		&= (N_x,0,0)\odot(0,0,S)-(N_x,0,0)^{\otimes 2} \\
		&= S (N_x,0,0)\odot(0,0,1)-(N_x,0,0)^{\otimes 2} \\
		&= S (0,0,1) \odot (N_x,0,0) - (N_x,0,0)^{\otimes 2}.
	\end{align*}
	
	Finally, note that $\wh \alpha = Sy \cdot v_1 + v_2 = 0 \in \wh H_0(\pi;\Gamma(\coker(d^2)))$, as required.
\end{proof}

\begin{proof}[Proof of \Cref{thm:coker=0}]
	\label{proof:coker=0}
	We showed previously that there was an exact sequence:
	\[ 
		\widehat{H}_0(\pi; \Gamma((N,2))) \xrightarrow[]{\widehat{i'_*}}
		\widehat{H}_0(\pi; \Gamma(\coker(d^2))) \xrightarrow[]{\widehat{q'}}
		\widehat{H}_0(\pi;E).
	\]
	By \Cref{prop:H_0(E)}, we have that $\wh H_0(\pi;E)=0$ and,
	by \Cref{lemma:i_+=0}, we have that $\widehat{i'_*}=0$.
	By exactness, this implies that $\widehat{H}_0(\pi; \Gamma(\coker(d^2)))=0$.
\end{proof}

\bibliographystyle{alpha}
\bibliography{dihedral_gamma}

\end{document}